\documentclass[12pt]{article}

\usepackage{amsmath}
\usepackage{amssymb}
\usepackage{amsfonts}
\usepackage{float}
\usepackage{epsfig}
\usepackage[usenames,dvipsnames]{xcolor}
\usepackage{graphicx}
\usepackage{caption}
\usepackage{subcaption}
\usepackage{placeins}
\usepackage{hyperref}
\usepackage{algorithm}
\usepackage{algorithmic}
\usepackage{calc}
\usepackage{tikz}
\usetikzlibrary{calc}

\allowdisplaybreaks[4]

\newtheorem{theorem}{Theorem}[section]

\newtheorem{corollary}{Corollary}[section]
\newtheorem{prop}{Proposition}[section]
\newtheorem{definition}{Definition}[section]

\newenvironment{proofof}[1]{\smallskip\noindent{\textbf{Proof~of~#1.}}%
  \hspace{1pt}}{\hspace{-5pt}{\nobreak\quad\nobreak\hfill\nobreak%
    $\square$\vspace{2pt}\par}\smallskip\goodbreak}

\newenvironment{proof}{\smallskip\noindent{\textbf{Proof.}}%
  \hspace{1pt}}{\hspace{-5pt}{\nobreak\quad\nobreak\hfill\nobreak%
    $\square$\vspace{2pt}\par}\smallskip\goodbreak}

\newcommand{\R}{\mathbb R}
\newcommand{\N}{{\mathbb N}}
\newcommand{\Z}{{\mathbb Z}}

\newcommand{\PRS}{{\cal PRS}}
\newcommand{\SPRS}{{\cal SPRS}}

\newcommand{\abs}[1]{\left\vert #1 \right\vert}

\newcommand{\Rr}{\mathcal{RS}}
\newcommand{\Rsol}{\mathcal{RS}}

\newcommand{\rhocr}{\rho_{\mathrm{cr}}}
\newcommand{\fmax}{f^{\max}}
\newcommand{\TV}{\mathrm{TV}}
\newcommand{\PV}{\mathrm{PV}}
\newcommand{\NV}{\mathrm{NV}}
\newcommand{\dx}{\Delta x}
\newcommand{\dt}{\Delta t}
\newcommand{\xj}[2]{x_{j#1\frac{#2}{2}}}

\DeclareMathOperator{\sgn}{sgn}

\begin{document}

\title{Priority-based Riemann solver for traffic flow on networks%
\thanks{This research was supported by the NSF grant CNS $\#$1446715,  by KI-Net "Kinetic description of emerging challenges in multiscale problems of natural sciences" - NSF grant \# 1107444 and by the INRIA associated team
‘Optimal REroute Strategies for Traffic managEment’ (ORESTE).}}
\author{Maria Laura Delle Monache%
\thanks{Department of Mathematical Sciences, Rutgers University - Camden, Camden, NJ, USA (email: \href{mailto:ml.dellemonache@rutgers.edu}{ml.dellemonache@rutgers.edu}).} 
\and Paola Goatin%
 \thanks{Inria Sophia Antipolis - M\'editerran\'ee, France (email: \href{mailto:paola.goatin@inria.fr}{paola.goatin@inria.fr}).} 
 \and Benedetto Piccoli%
 \thanks{Department of Mathematical Sciences and CCIB, Rutgers University - Camden, Camden, NJ, USA (email: \href{mailto:piccoli@camden.rutgers.edu}{piccoli@camden.rutgers.edu}).}
 }
\date{}
\maketitle
\begin{abstract}
In this article we introduce a new Riemann solver for traffic flow on networks. 
The Priority Riemann solver ($\PRS$) provides a solution at junctions by taking into consideration priorities for the incoming roads and maximization of through flux. 
We prove existence of solutions for the solver for junctions with up to two incoming
and two outgoing roads and
show numerically the comparison with previous Riemann solvers. 
Additionally, we introduce a second version of the solver that considers the priorities as softer constraints and illustrate numerically the differences between the two solvers. 
\textbf{AMS sybject classifications}: 90B20, 35L65
\end{abstract}

\section{Introduction}
Conservation law on network is now a mature field with an increasing number of contributions
in recent years. The theory for the scalar case is quite developed (see \cite{CGP05,GP09,MR1338371}),
with most results based on the concept of Riemann solver. The latter is the network equivalent
to the classical Riemann solvers for conservation laws on the real line and provide
a solution to Riemann problems at junctions, i.e., Cauchy problems with constant initial data
on each road.\\
This theory was applied to different domains, including vehicular traffic \cite{GPbook},
supply chains \cite{ADR06}, irrigation channels \cite{NHM-SI09} and others. For a complete
account of recent results and references we refer the reader to the survey 
\cite{MR3200227}.

For vehicular traffic, authors considered many different traffic situations to be modeled,
thus proposing a rich set of alternative junction models even for the scalar case, 
see \cite{CP05,CGP05,DP08,DMRSKGB14,GPbook,MR1338371,MP08,SRKDMGBpreprint}.
Here, we first propose a new model which considers priorities among the incoming roads
as the first criterion and maximization of flux as the second.
The main idea is that the road with the highest priority will use the maximal flow
taking into account also outgoing roads constraints. If some room is left for additional flow
then the road with the second highest priority will use the left space and so son.
A precise definition of the new Riemann solver, called \emph{Priority Riemann Solver},
is based on a traffic distribution matrix $A$ (Definition \ref{eq:calA}),
a priority vector $P=(p_1,\ldots,p_n)$ (with $p_i\geq 0$ and $\sum_i p_i=1$) and
requires a recursion method, which is described in Algorithm \ref{alg:PRS}.\\
We also model special situations in which some outgoing roads do not absorb traffic from some
incoming ones and propose an alternative solver with softer priorities, see Algorithm \ref{alg:SPRS}.\\
During the writing of this manuscript we discovered that our priority-based Riemann solver may be
obtained as limit of solvers defined by Dynamic Traffic Assignment based on junctions with queues
\cite{BN16}.

The general existence theorem of \cite{GP09} can be applied to every Riemann solver
satisfying three general properties, called (P1)-(P3), but
can not be applied in the present case. Indeed, the proof is based on
estimates on the flow total variation in space on the network in terms of the total variation
in time of the flow through the junction $\Gamma$, see definition \eqref{eq:gamma}.
In turn the latter is bounded thanks to the general property (P3),
which ensures that waves bringing flux decrease to the junction provoke a decrease of $\Gamma$.
Such property (P3) is not satisfied by the Priority Riemann Solver
(see the Appendix: case A2 with flux increase corresponding to Figure \ref{fig:CaseA2_increasing}).\\
Therefore, we achieve existence via a new set of general properties. Property (P1)
is the same as that of \cite{GP09}, while we modify (P2) and (P3) by using estimates
involving not only $\Gamma$ but also the maximal flow along the priority vector $P$
in the set of admissible flows, see definition \eqref{eq:hbar}.\\
Then we apply the general theory to the Priority Riemann Solver by proving that the new
(P1)-(P3) are satisfied for junctions with at most two incoming and 
two outgoing roads.\\
Then, to illustrate the Priority Riemann Solver, the one with soft priorities and compare
with existing ones, we implement numerical simulations via the Godunov scheme. 

The paper is organized as follows. In Section \ref{sec:basics} we introduce the basic definitions of the theory of conservation laws on networks, then in Section \ref{sec:PRS} we define our Priority Riemann Solver and prove existence of solutions to Cauchy problems in Section \ref{sec:ex_result}. In Section \ref{sec:SPRS}, an alternative definition of the Riemann Solver with softer priorities is described and lastly, in Section \ref{sec:num_sim}, we propose a numerical discretization and show some numerical simulations comparing our Solvers to existing ones. The Appendix \ref{sec:appendix} collects the proof of the main theorem of the paper.


\section{Basics}\label{sec:basics}
In this section we recall the basic definitions and results
of the theory of conservation laws on networks, based
on the concept of Riemann solver at junctions.
Due to finite  propagation speed of waves, to achieve existence results for Cauchy problems
it is not restrictive to focus on a single junction. For details on how to extend
the results to a general network, we refer the reader to \cite{GPbook,GP09} .

Fix a junction $J$ with $n$ incoming roads $I_1,\ldots,I_n$ and $m$ outgoing roads $I_{n+1},\ldots,I_{n+m}$, where
$I_i=]-\infty,0]$ ($i\in\{1,\ldots,n\}$)
and $I_j=[0,+\infty[$ ($j\in\{n+1,\ldots,n+m\}$).
The traffic on each road $I_l$ ($l\in\{1,\ldots,n+m\}$) is modeled
using the celebrated Lighthill-Whitham-Richards model (briefly LWR, see \cite{LW55,R56}):
\begin{equation}
  \label{eq:LWR}
\partial_t  \rho_l+ \partial_xf(\rho_l)=0,\quad t\geq 0, x\in I_l ,
\end{equation}
where $\rho_l(t,x)\in [0,\rho_{\max}]$,
is the car density, $v_l=v_l(\rho_l)$
is the average velocity and $f_l=f_l(\rho_l)=\rho_l v_l(\rho_l)$
is the flux.
For simplicity, throughout the paper we assume $\rho_{\max}=1$ and $f_l=f$ for all $l=1,\ldots, n+m$.

We make the following assumptions on the flux function $f$:
\begin{itemize}
\item[{\bf (H)}]
  $f : [0,1] \rightarrow \R$ is a Lipschitz continuous and
  concave function satisfying
  \begin{enumerate}
  \item $f(0)=f(1)=0$;

  \item there exists a unique $\rhocr\in \ ]0,1[$ such that $f$ is strictly
    increasing in $[0,\rhocr[$ and strictly decreasing in $]\rhocr,1]$.
  \end{enumerate}
\end{itemize}
As usual, entropic solutions and weak solutions at junctions
are given by:
\begin{definition}
  A function $\rho_l\in C([0,+\infty[;L^1_{loc}(I_l))$
  is an entropy-admissible
  solution to~(\ref{eq:LWR}) in the arc $I_l$ if, 
  for every $k\in [0,1]$ and every $\varphi: [0,+\infty[\times I_l\to\R$ 
  smooth, positive and with compact support in
  $]0,+\infty[ \times \left( I_l \setminus \{0\} \right)$, it holds
  \begin{equation} \label{eq:entsol-onearc}
    \int_0^{+\infty}\int_{I_l}\Big( |\rho_l -k| \partial_t \varphi
      + \sgn(\rho_l-k)(f(\rho_l)- f(k))
    \partial_x \varphi \Big)dxdt \geq 0.
  \end{equation}
\end{definition}

\begin{definition}\label{def:weak_solution}
  A collection of functions $\rho_l\in C([0,+\infty[;L^1_{loc}(I_l))$,
  ($l\in\{1,\ldots,n+m\}$) is a weak solution at $J$ if
  \begin{enumerate}
  \item for every $l\in\{1,\ldots,n+m\}$, the function
    $\rho_l$ is an entropy-admissible solution to~(\ref{eq:LWR})
    in the road $I_l$;

  \item for every $l\in\{1,\ldots,n+m\}$ and for a.e. $t>0$, the function
    $x\mapsto\rho_l(t,x)$ has a version with bounded total variation;

  \item for a.e. $t>0$, it holds
    \begin{equation}\label{eq:RH}
      \sum\limits_{i=1}^n f(\rho_i (t, 0-)) = \sum\limits_{j=n
        +1}^{n+m}f(\rho_j (t, 0+))\,,
    \end{equation}
    where $\rho_l$ stands for the version with bounded total variation of 2.
  \end{enumerate}
\end{definition}

A Riemann problem at the junction $J$ is a Cauchy
problem with constant initial data on each road.
More precisely, given $\rho_{1,0},\ldots,\rho_{n+m,0}\in[0,1]$, the corresponding Riemann problem at $J$ is given by
\begin{equation}
  \label{eq:RPatJ}
  \left\{
    \begin{array}{ll}
      \begin{array}{l}
        \partial_t \rho_l+\partial_x f(\rho_l)=0,
        \vspace{.2cm}\\
        \rho_l(0,\cdot)=\rho_{0,l}, 
      \end{array}
      & l\in\{1,\ldots,n+m\}.
    \end{array}
  \right.
\end{equation}

For a collection of functions $\rho_l \in C([0,+\infty[; L^1_{loc} (I_l))$
($l\in \{ 1,\ldots, n+m\}$) such that, for every $l\in \{1,\ldots, n+m\}$
and a.e. $t>0$, the map $x\mapsto \rho_l(t,x)$ has a version with bounded
total variation, we define the functionals
\begin{equation}\label{eq:gamma}
  \Gamma(t):=\sum_{i=1}^{n} f(\rho_{i}(t,0-))
\end{equation}
and
\begin{equation}
  \label{eq:fun2}
  \TV_f(t) := \sum_{l=1}^{n+m} \TV 
  \left(f \left( \rho_{l}(t,\cdot) \right) \right).
\end{equation}
\noindent
Notice that $\Gamma$ is the flux through the junction, i.e. the total number of
cars crossing the junction $J$ per unit of time, while $\TV_f$ is the total variation
of the flux on the whole network.
From the flux bounds we easily derive:
\begin{equation}
  \label{eq:gamma_bound}
  0\le\Gamma(t)\le nf(\rhocr).
\end{equation}

%
%
A Riemann solver at $J$ is defined by:
\begin{definition}\label{def:Riemann_solver}
  A Riemann solver $\mathcal{RS}$ is a function
  \begin{equation*}
    \begin{array}{rccc}
      \mathcal{RS}: & [0,1]^{n+m} & \longrightarrow
      & [0,1]^{n+m}\\
      & (\rho_{1,0},\ldots,\rho_{n+m,0}) & \longmapsto & 
      (\bar\rho_1,\ldots,\bar\rho_{n+m})
    \end{array} 
  \end{equation*}
  satisfying the following properties
  \begin{enumerate}
  \item $\sum\limits_{i=1}^nf(\bar\rho_i)=\sum\limits_{j=n+1}^{n+m}f(\bar\rho_j)$;

  \item for every $i\in\{1,\ldots,n\}$, the classical Riemann problem
    \begin{equation*}
      \left\{
        \begin{array}{l}
          \rho_t+f(\rho)_x=0,\hspace{1cm}x\in\R,\, t>0,\vspace{.2cm}\\
          \rho(0,x)=\left\{
            \begin{array}{ll}
              \rho_{i,0}, & \textrm{ if } x<0,\\
              \bar\rho_i, & \textrm{ if } x>0,
            \end{array}
          \right.
        \end{array}
      \right.
    \end{equation*}
    is solved with waves with negative speed;

  \item for every $j\in\{n+1,\ldots,n+m\}$, the classical Riemann problem
    \begin{equation*}
      \left\{
        \begin{array}{l}
          \rho_t+f(\rho)_x=0,\hspace{1cm}x\in\R,\, t>0,\vspace{.2cm}\\
          \rho(0,x)=\left\{
            \begin{array}{ll}
              \bar\rho_j, & \textrm{ if } x<0,\\
              \rho_{j,0}, & \textrm{ if } x>0,
            \end{array}
          \right.
        \end{array}
      \right.
    \end{equation*}
    is solved with waves with positive speed.
  \end{enumerate}
Moreover, the Riemann solver $\Rsol$ must satisfy the consistency condition if
  \begin{equation*}
    \Rsol(\Rsol(\rho_{1,0},\ldots,\rho_{n+m,0}))=
    \Rsol(\rho_{1,0},\ldots,\rho_{n+m,0})
  \end{equation*}
  for every $(\rho_{1,0},\ldots,\rho_{n+m,0})\in[0,1]^{n+m}$.
\end{definition}

%
%

%
%

For future use, we now provide some definitions
for the LWR model and for Riemann problems at junctions, for more details see \cite{GPbook}.
\begin{definition}\label{def:equilibrium}
  We say that $(\rho_{1,0},\ldots,\rho_{n+m,0})$ is an equilibrium for the
  Riemann solver $\Rsol$ if
  \begin{equation*}
    \Rsol(\rho_{1,0},\ldots,\rho_{n+m,0})=(\rho_{1,0},\ldots,\rho_{n+m,0}).
  \end{equation*}
\end{definition}
\begin{definition}\label{def:good-bad}
  We say that a datum $\rho_i\in[0,1]$ in an incoming road is a good datum if
  $\rho_i\in[\rhocr,1]$ and a bad datum otherwise.

  We say that a datum $\rho_j\in[0,1]$ in an outgoing road is a good datum if
  $\rho_i\in[0,\rhocr]$ and a bad datum otherwise.
\end{definition}
We also define the following function:
\begin{definition} \label{def:tau}
  Let $\tau:[0,1] \rightarrow [0,1]$ be the map such that:
  \begin{enumerate}
  \item $f(\tau(\rho))=f(\rho)$ for every $\rho\in[0,1]$;
    
  \item $\tau(\rho) \not= \rho$ for every $\rho\in[0,1]\setminus\{\rhocr\}$.
  \end{enumerate}
\end{definition}
Clearly, the function $\tau$ is well defined and satisfies
\begin{equation*}
  0\le \rho \le \rhocr\Longleftrightarrow \rhocr
  \le \tau(\rho) \le 1,\qquad \rhocr\le \rho \le 1\Longleftrightarrow 0 \le
  \tau(\rho) \le \rhocr.
\end{equation*}
\noindent
Given initial data (of Riemann type)
$\rho_{1,0},\ldots,\rho_{n+m,0}\in[0,1]$ we define:
%
\begin{enumerate}
\item for every $i\in\{1,\ldots,n\}$ 
  \begin{equation}
    \label{eq:gamma_i}
    \gamma^{max}_i=\left\{
      \begin{array}{ll}
        f(\rho_{i,0}), & \textrm{ if } 0\le\rho_{i,0}\le\rhocr,
        \vspace{.2cm}\\
        { }f(\rhocr), & \textrm{ if } \rhocr\le\rho_{i,0}\le1;
      \end{array}
    \right.
  \end{equation}

\item for every $j\in\{n+1,\ldots,n+m\}$
  \begin{equation}
    \label{eq:gamma_j}
    \gamma^{max}_j=\left\{
      \begin{array}{ll}
        f(\rhocr), & \textrm{ if } 0\le\rho_{j,0}\le\rhocr,
        \vspace{.2cm}\\
        { }f(\rho_{j,0}), & \textrm{ if } \rhocr\le\rho_{j,0}\le1;
      \end{array}
    \right.
  \end{equation}

\item for every $l\in\{1,\ldots,n+m\}$ 
  \begin{equation}
    \label{eq:omega_l}
    \Omega_l=[0,\gamma^{max}_l].
  \end{equation}
\end{enumerate}
\noindent
Moreover, we have the following result (see \cite{GP09}):
\begin{prop}
  It holds:
  \begin{enumerate}
  \item For every $i\in\{1,\ldots,n\}$, an element $\bar\gamma$ belongs to
    $\Omega_i$ if and only if there exists $\bar\rho_i\in[0,1]$ such that
    $f(\bar\rho_i)=\bar\gamma$ and point 2 of
    Definition~\ref{def:Riemann_solver} is satisfied.

  \item For every $j\in\{n+1,\ldots,n+m\}$, an element $\bar\gamma$ belongs to
    $\Omega_j$ if and only if there exists $\bar\rho_j\in[0,1]$ such that
    $f(\bar\rho_j)=\bar\gamma$ and point 3 of
    Definition~\ref{def:Riemann_solver} is satisfied.
  \end{enumerate}
\end{prop}

\section{Definition of the Priority Riemann Solver}\label{sec:PRS}

In this section we define a new Riemann solver
based on priorities. For this purpose,
we first fix a matrix $A$ belonging to the set of matrices:
\begin{equation}
  \label{eq:calA}
  \mathcal A:=\left\{
    \begin{array}{ll}
      A=\{a_{ji}\}_{\substack{i=1,\ldots,n\\ j=n+1,\ldots,n+m}}: &
      \begin{array}{l}
        0 \leq a_{ji} \leq   1\,\,      \forall i,j,\\
        \sum\limits_{j=n +1}^{n+m} a_{ji} =1\,\,\forall i
      \end{array}    
    \end{array}
  \right\}
\end{equation}
and a priority vector $P=(p_1,\ldots,p_n)\in\R^n$,
with $p_i>0$, $\sum_i p_i=1$,
indicating priorities among incoming roads.\\
Consider the closed, convex and non-empty set
  \begin{equation}
    \label{eq:omega}
    \Omega=\left\{
      (\gamma_1,\cdots,\gamma_n)\in\prod_{i=1}^n\Omega_i:
      A\cdot (\gamma_1,\cdots,\gamma_n)^T\in\prod_{j=n+1}^{n+m}\Omega_j
    \right\}\, ,
  \end{equation}
and define:
\begin{equation}
	\label{eq:hbar}
	\bar{h} = \sup \{ h\in \R^+ : hP \in \Omega\}.
\end{equation}

Given Riemann data $(\rho_{1,0},\ldots,\rho_{m+n,0})$,
we define a vector $Q=(\bar\gamma_1,\ldots,\bar\gamma_n)$ 
of incoming fluxes by a recursive procedure. First we explain
the procedure in steps and then provide a pseudo-code in Algorithm \ref{alg:PRS}.

\begin{itemize}
\item {\bf STEP 1.} For every $i\in\{1,\ldots,n\}$ define
\[
h_i= \max \{h:h\,p_i\leq \gamma_i^{max}\}=\frac{\gamma_i^{max}}{p_i},
\]
and for every $j\in\{n+1\ldots,n+m\}$ define
\[
h_j= \max \left\{h:(A\cdot(h P))_j=h\left(\sum_i a_{ji}p_i\right)\leq \gamma_j^{max}\right\}=\frac{\gamma_j^{max}}{\sum_i a_{ji}p_i}.
\]
In other words, $h_i$ is the maximal $t$ so that $hP$ verifies the flux constraint for the $i$-th road, similarly for $h_j$.\\
Set $\hbar=\min_{ij}\{h_i,h_j\}$.

We distinguish two cases:
\begin{itemize}
\item CASE 1. If there exists $j$ such that $h_j=\hbar$, then we set $Q=\hbar \,P$ and we are done.
\item CASE 2. Otherwise, let $I_1=\{i\in\{1,\ldots,n\}:h_i=\hbar\}$ (by assumption $I_1\not=\emptyset$).
We set $Q_i=\hbar \,p_i$  for $i\in I_1$ and we go to next step.
\end{itemize}
\item {\bf STEP S.} In step $S-1$ we defined a set $I_{S-1}$ and, by induction, all components
of $Q$ are fixed for $i\in J_S=I_1\cup\cdots \cup I_{S-1}$. We let $|J_S|<n$ denote the cardinality of $J_S$
and denote by $J_S^c$ the complement of $J_S$ in $\{1,\ldots,n\}$.
We now define $h_i$ for $i\in J_S^c$ by:
\[
h_i= \max \{h:h\,p_i\leq \gamma_i^{max}\}=\frac{\gamma_i^{max}}{p_i},
\]
and for every $j\in\{n+1\ldots,n+m\}$ define
\[
h_j= \max \left\{h:\ \sum_{i\in J_S} a_{ji}Q_i+h\left(\sum_{i\in J_S^c}a_{ji}p_i\right)\leq \gamma_j^{max}\right\}.
\]
We then proceed similarly to STEP 1, setting $\hbar=\min_{ij}\{h_i,h_j\}$ and distinguishing two cases:
\begin{itemize}
\item CASE 1. If there exists $j$ such that $h_j=\hbar$, then we set $Q_i=\hbar  \,P_i$ for $i\in J_S^c$
and we are done.
\item CASE 2. Otherwise, let $I_S=\{i\in J_S^c:h_i=\hbar\}$ (by assumption $I_S\not=\emptyset$).
We set $Q_i=\hbar \,p_i$  for $i\in I_S$. If $J_S\cup I_S=\{1,\ldots,n\}$ then we stop, otherwise we go to next step.
\end{itemize}
\end{itemize}
\begin{algorithm}[H]
\begin{algorithmic}
\STATE Set $J=\emptyset$ and $J^c=\{1,\ldots,n\}\setminus J$.
\WHILE{$\abs{J} < n$}
\STATE $\forall i\in J^c$ $\rightarrow$  $h_i= \max \{h:h\,p_i\leq \gamma_i^{max}\}=\frac{\gamma_i^{max}}{p_i},$
\STATE $\forall j\in\{n+1\ldots,n+m\}$ $\rightarrow$  $h_j= \sup \{h:\sum_{i\in J} a_{ji} Q_i +  h(\sum_{i\in J^c} a_{ji}p_i)\leq \gamma_j^{max}\}$.
\STATE Set $\hbar =\min_{ij}\{h_i,h_j\}$.
\IF {$\exists\ j $ s.t. $h_j=\hbar$ }
	\STATE Set $Q=\hbar \,P$ and $J=\{1,\ldots, n\}$.
\ELSE 
\STATE Set $I=\{i\in J^c:h_i=\hbar\}$ and $Q_i=\hbar\,p_i$  for $i\in I$.
\STATE Set $J=J\cup I$.
\ENDIF
\ENDWHILE
\caption{Recursive definition of $\PRS$}
\label{alg:PRS}
\end{algorithmic}
\end{algorithm}
We are now ready to define the Priority Riemann Solver 
(briefly $\PRS$).
\begin{definition}
Let $Q=(\bar\gamma_1,\ldots, \bar\gamma_n)$ be the vector of incoming fluxes defined by 
Algorithm \ref{alg:PRS}, then the vector of outgoing fluxes is 
given by $A\cdot Q^T= (\bar\gamma_{n+1},\ldots, \bar\gamma_{n+m})^T$.\\
For every $i\in\{1,\ldots,n\}$, set $\bar\rho_i$ equal either to $\rho_{i,0}$ if $f(\rho_{i,0})=\bar\gamma_i$, or to the solution to $f(\rho)=\bar\gamma_i$ such that $\bar\rho_i\ge\rhocr$.
For every $j\in\{n+1,\ldots,n+m\}$, set $\bar\rho_j$ equal either to $\rho_{j,0}$ if $f(\rho_{j,0})=\bar\gamma_j$, or to the solution
to $f(\rho)=\bar\gamma_j$ such that $\bar\rho_j\le\rhocr$. Finally, $\PRS:[0,1]^{n+m}\to[0,1]^{n+m}$ is given by
  \begin{equation}\label{eq:rs1_rho}
    \PRS(\rho_{1,0},\ldots,\rho_{n+m,0})
    =(\bar\rho_1,\ldots,\bar\rho_n,\bar\rho_{n+1},\ldots,\bar\rho_{n+m})\,.
  \end{equation}
\end{definition}

\section{Existence result for Cauchy problems}\label{sec:ex_result}
Given initial data of bounded variation $\rho_{0,1}(\cdot),\ldots,\rho_{0,n}(\cdot):[0,+\infty[\to [0,1]$
and $\rho_{0,n+1}(\cdot),\ldots,\rho_{0,n+m}(\cdot):]-\infty,0]\to [0,1]$
the corresponding Cauchy problem is defined by:
\begin{equation}
  \label{eq:CP}
  \left\{
    \begin{array}{ll}
      \begin{array}{l}
        \partial_t \rho_l+\partial_x f(\rho_l)=0,
        \vspace{.2cm}\\
        \rho_l(0,x)=\rho_{0,l}(x), 
      \end{array}
      & l\in\{1,\ldots,n+m\}.
    \end{array}
  \right.
\end{equation}
To solve Cauchy problems one can construct approximate solutions via Wave Front Tracking (WFT).
In simple words, one first approximate the initial data by piecewise constant functions,
then solve the corresponding Riemann problems within roads and at junctions approximating
rarefaction waves by a fan of rarefaction shocks and solve new Riemann problems when
waves interact with each other or with the junction. We refer the reader to \cite{GPbook} for details.
Notice that all waves in a WFT approximate solution are shocks, i.e. traveling discontinuities.
For every wave we will usually indicate by $\rho_l$, respectively $\rho_r$, the left limit,
respectively right limit, of the approximate solution at the discontinuity point.
To prove convergence of WFT approximations, one needs to estimate the number of waves, the number of 
wave interactions and provide estimates on the total variation of approximate solutions.
The general theory of \cite{GP09} is based on three properties which guarantee such estimates.
Along the same idea we define three general properties (P1)-(P3) which will ensure existence
of solutions.

The first property requires that equilibria are determined
only by bad data values (and coincides with (P1) in \cite{GP09}), more precisely:
%
%
\begin{definition}\label{def:Prop_P1}
  We say that a Riemann solver $\Rsol$ has the property (P1)
  if the
  following condition holds.
  Given $(\rho_{1,0},\ldots,\rho_{n+m,0})$ and 
  $(\rho'_{1,0},\ldots,\rho'_{n+m,0})$ two initial data such that
  $\rho_{l,0}=\rho'_{l,0}$ whenever either $\rho_{l,0}$ or $\rho'_{l,0}$
  is a bad datum, then
  \begin{equation}
    \Rsol(\rho_{1,0},\ldots,\rho_{n+m,0})=\Rsol(\rho'_{1,0},\ldots,\rho'_{n+m,0}).
  \end{equation}
\end{definition}

The second property requires for bounds in the increase of the flux
variation for waves interacting with $J$.
More precisely the latter is bounded in terms of the
strength of the interacting wave as well
as the sum of the changes in the incoming fluxes and in $\bar h$ (see (\ref{eq:hbar})).
Moreover, the increase in $\bar h$ is bounded by the strength of the interacting wave.
%
%
\begin{definition}\label{def:Prop_P2}
  We say that a Riemann solver $\Rsol$ has the property (P2) if
  there exists a constant $C \ge 1$ such that the following condition
  holds.
  For every equilibrium $(\rho_{1,0}, \ldots, \rho_{n+m,0})$ of $\Rr$
  and for every wave $(\rho_{i},\rho_{i,0})$ for $i=1,\ldots,n$
(respectively $(\rho_{j,0},\rho_{j})$ for $j=n+1,\ldots,n+m$)
  interacting with $J$ at time $\bar t > 0$ and producing waves in the
  arcs according to $\Rr$, we have
  \begin{equation}
    \begin{array}{c}
      \TV_f (\bar t+) - \TV_f (\bar t-)\\
      \le C \min \left \{
        \abs{f(\rho_{l,0}) - f(\rho_l)}, 
        \abs{\Gamma(\bar t+) - \Gamma(\bar t-)}+\abs{\bar{h}(\bar{t}+)-\bar{h}(\bar{t}-)}
      \right \}
    \end{array}
  \end{equation}
and
\begin{equation}
	\label{eq:deltaT}
	\bar{h}(\bar{t}+)-\bar{h}(\bar{t}-) \leq C \abs{f(\rho_{l,0})-f(\rho_l)}.
\end{equation}
\end{definition}

Finally, we state the third property: if a wave interacts with $J$ and provokes a flux decrease then 
$\bar h$ decreases and the increase of $\Gamma$ is bounded by the change in  $\bar h$.
\begin{definition}\label{def:Prop_P3}
  We say that a Riemann solver $\Rsol$ has the property (P3) if the following holds.
  For every equilibrium
  $(\rho_{1,0}, \ldots, \rho_{n+m,0})$ of $\Rr$
  and for every wave $(\rho_{i},\rho_{i,0})$ with $f(\rho_i) < f(\rho_{i,0})$ for $i=1,\ldots,n$
(respectively $(\rho_{j,0},\rho_{j})$ with $f(\rho_j) < f(\rho_{j,0})$ for $j=n+1,\ldots,n+m$)
  interacting with $J$ at time $\bar t > 0$ and producing waves in the
  arcs according to $\Rr$, we have
  \begin{equation}
    \Gamma (\bar t+) - \Gamma (\bar t-)\leq C \abs{\bar{h}(\bar{t}+)-\bar{h}(\bar{t}-)},
  \end{equation}
  \begin{equation}
    \bar{h}(\bar{t}+)\leq \bar{h}(\bar{t}-).
  \end{equation}
\end{definition}

\begin{theorem}
\label{th:CPsol}
If a Riemann solver satisfies (P1)-(P3), then
every Cauchy problem with initial data of bounded variation admits a weak solution.
\end{theorem}

In order to prove Theorem \ref{th:CPsol}, we need first to provide some definition
and results. We start by giving the following:
\begin{definition}
A wave along a WFT approximate solution generated at time $t=0$ inside a road
is called original, while the ones generated by $J$ are called not original. 
If two original waves interact, then the resulting wave is still
called original, 
while if an original wave interacts with a not original wave
then the resulting wave is not original.
\end{definition}
We now have the following result:
\begin{prop}\label{prop:returning-waves}
Let $(\rho_l,\rho_r)$ be a wave generated on an incoming road $I$ from the junction at time $\bar s$.
Assume that there exists a time $\bar t>\bar s$ at which the wave interacts with $J$
(after interacting with waves inside $I$) and call $\bar\rho_l$, respectively
$\bar\rho_r$, its left, respectively, right limit at $\bar t-$. 
If $I$ is an incoming road then
we have $\bar\rho_r\geq \rhocr$ and $f(\bar\rho_l)<f(\bar\rho_r)$.
If $I$ is an outgoing road then
we have $\bar\rho_l\leq \rhocr$ and $f(\bar\rho_r)<f(\bar\rho_l)$.
\end{prop}

\begin{proof}
We prove the result for an incoming road, the other case being similar.
First notice that $(\rho_l,\rho_r)$ must have negative speed,
thus if $\rho_l<\rhocr$ then $\rho_r>\tau(\rho_r)\geq\rhocr$, while if $\rho_l\geq\rhocr$
then $\rho_r\geq\rhocr$. Therefore in both cases we have $\rho_r\geq\rhocr$.
If the wave interacts with waves coming from the left then the value of $\rho_r$ does not change.
If the wave interacts with a wave $(\rho_r,\hat\rho)$ coming from the right, then the wave
was generated from the junction $J$ (or obtained by interactions of waves generated from $J$) 
and thus must satisfy $\hat\rho\geq\rhocr$. Finally, $\bar\rho_r\geq\rhocr$.
Then, since $(\bar\rho_l,\bar\rho_r)$ must have positive speed,
we deduce that $\bar\rho_l<\rhocr$ and $f(\bar\rho_l)<f(\bar\rho_r)$, thus we conclude.
\end{proof}
From Proposition \ref{prop:returning-waves} we have the following:
\begin{corollary}\label{cor:decr_gamma}
If a wave $(\rho_l,\rho_r)$ interacts with $J$ from an incoming road
and satisfy $f(\rho_l)>f(\rho_r)$ then it is an original wave.
If a wave $(\rho_l,\rho_r)$ interacts with $J$ from an outgoing road
and satisfy $f(\rho_r)>f(\rho_l)$ then it is an original wave.
\end{corollary}

\begin{proofof}{Theorem \ref{th:CPsol}} 
WFT approximate solutions can be constructed because (P1) holds true (see \cite{GP09}).

We now prove that the total variation of the flux $\TV(f)$ remains uniformly bounded in time
along WFT approximate solutions.
The main idea is to first bound the total variation
in time of $\bar{h}$ and then of $\Gamma$.
This in turn will provide the desired estimate.\\
Let us indicate with $\PV$ the positive variation of a function
and with $\NV$ the negative one. Then:
\[
\TV(\bar{h})=\PV(\bar{h})+\NV(\bar{h})
\]
and
\[
\PV(\bar{h})=\PV^{O}(\bar{h})+\PV^{R}(\bar{h}),
\]
where $\PV^{O}$ is the variation due to interactions
of original waves with the junction and $\PV^{R}$
the one due to returning waves (i.e. not original).\\
From (P2) we get:
\[
\PV^{O}(\bar{h})\leq C \TV(f(\rho_0))
\leq C\ \max_{\rho\in [0,1]} \abs{f'(\rho)}\ \TV(\rho_0)
\]
and from (P3) and Corollary \ref{cor:decr_gamma}
it follows $\PV^{R}(\bar{h})=0$.
Then $\PV(\bar{h})$ is bounded and,
since $\bar{h}\leq \fmax/\max_i p_i$,
also $\TV(\bar{h})$ is bounded.
Similarly, for $\Gamma$ we can write
\[
\TV(\Gamma)=\PV(\Gamma)+\NV(\Gamma)
\]
and
\[
\PV(\Gamma)=\PV^{O}(\Gamma)+\PV^{R}(\Gamma).
\]
Following the proof of \cite[Lemma 12]{GP09}, (P2) implies
\[
\PV^{O}(\Gamma)\leq (C+2)\TV(f(\rho_0))\leq (C+2)\ \max_{\rho\in [0,1]} \abs{f'(\rho)}\ \TV(\rho_0).
\]
From (P3) we have that 
\[
\PV^{R}(\Gamma) \leq C \TV(\bar{h}) ,
\]
which we just proved to be bounded. 
Therefore $\TV(\Gamma)$ is also bounded.\\
Now, define $\mathrm{Int}$ the set of times at which a wave
interacts with the junction $J$ and, for $s\in\mathrm{Int}$
let us indicate by $\Delta\TV_f(s)$ the change due to the interaction.
By (P2) we have
\[ 
\TV(f(t))\leq \TV(f(\rho_0))+\sum_{s\in\mathrm{Int},s\leq t} 
\Delta \TV_f(s)
\]
\[
\leq  \max_{\rho\in [0,1]} \abs{f'(\rho)}\ \TV(\rho_0)+ C(\TV(\Gamma)+\TV(\bar{h})).
\]

Once $\TV(f)$ is bounded, one can obtain a bound on $\TV(\rho)$
as in \cite{GP09} and conclude by passing to the limit in WFT approximate solutions.
\end{proofof}
 
\begin{prop}\label{prop:P123}
The Priority Riemann Solver $\PRS$ satisfies (P1)-(P3) for junctions with $n\leq 2$, $m\leq 2$ and $0<a_{ji}<1$ for all $i,j$.
\end{prop}
The technical proof is deferred  to Section \ref{sec:appendix}.

\section{Solver with softer priorities}\label{sec:SPRS}
%

In this section we define a different version of the Riemann solver
that uses priorities as softer constraints. In particular, this solver will differ from the solver $\PRS$ defined in Section \ref{sec:PRS} when one of the entries of the matrix $A$, defined in \eqref{eq:calA}, vanishes, see Figure \ref{fig:diff_solver}. 
Notice that the softer priority of the $\SPRS$ will allow
some flow from road $2$ to pass through the junction,
when the maximal flow from road $1$ is already reached.
This reflects the situation where the physical geometry
of the junction allows for traffic from road $2$ to road $4$
(no traffic goes from $2$ to $3$) even if the traffic from
road $1$ to road $3$ is maximal and has higher priority.
\begin{figure}[ht]
\begin{subfigure}[H]{0.4\textwidth}
	\centering
	\begin{tikzpicture}
\draw[<->](5,0)--(0,0)--(0,5);
\draw(4,0)--(4,3)--(0,3);
\draw(2,0)--(2,4.5);
\draw(0,0)--(5,3);
\draw(0,4.5)--(5,1.5);
\draw[fill=gray](2,1.2)circle(0.1);
\node[below]at(5,0){$\gamma_1$};
\node[left]at(0,5){$\gamma_2$};
\node[below]at(5,3){$P$};
\node[left]at(2,1.3){$Q$};
\node[below]at(5,2){$\gamma_4$};
\node[right]at(2,4.5){$\gamma_3$};
\end{tikzpicture}
	\caption{Priority Riemann Solver}
	\label{fig:PRS}
\end{subfigure}\hspace{2cm} 
\begin{subfigure}[H]{0.4\textwidth}
	\centering
	\begin{tikzpicture}
\draw[<->](5,0)--(0,0)--(0,5);
\draw(4,0)--(4,3)--(0,3);
\draw(2,0)--(2,4.5);
\draw(0,0)--(5,3);
\draw(0,4.5)--(5,1.5);
\draw[fill=gray](2,3)circle(0.1);
\node[below]at(5,0){$\gamma_1$};
\node[left]at(0,5){$\gamma_2$};
\node[below]at(5,3){$P$};
\node[left]at(2,2.8){$Q$};
\node[below]at(5,2){$\gamma_4$};
\node[right]at(2,4.5){$\gamma_3$};
\end{tikzpicture}
	\caption{Solver with softer priorities}
	\label{fig:SPRS}
\end{subfigure}
\caption{Different solution for the two solvers: $\PRS$
and $\SPRS$. The distribution matrix $A$ satisfies
$a_{13}=0$, i.e. no cars enter road $3$ from road $1$.
Correspondingly the constraint on the flux $\gamma_3$
is represented by a vertical line. 
The $\PRS$ selects the point $Q$ in (a), while $\SPRS$
the point $Q$ in (b). Since priorities are softer
the flux through the junction of the $\SPRS$ solution
is higher than that of the $\SPRS$.
}
\label{fig:diff_solver}
\end{figure}
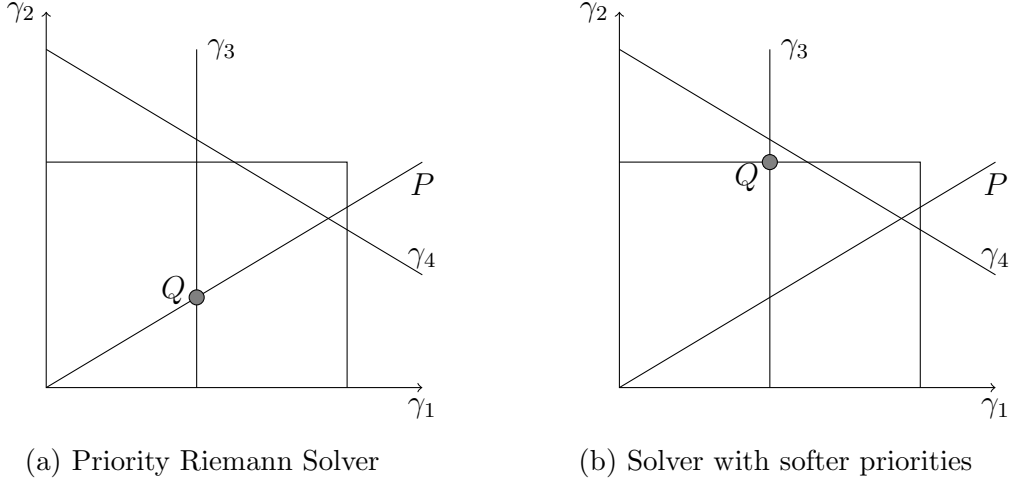
For this purpose,
we consider a matrix $A$ that 
may have $a_{ji} = 0$ for some $i,j$
and a priority vector $P=(p_1,\ldots,p_n)\in\R^n$,
with $p_i>0$, $\sum_i p_i=1$.\\
Then the Riemann solver with softer priorities (briefly $\SPRS$) can be defined by the following recursive algorithm:
\begin{algorithm}[H]
\begin{algorithmic}
\STATE Set $J=\emptyset$ and $J^c=\{1,\ldots,n\}\setminus J$.
\WHILE{$\abs{J} < n$}
\STATE $\forall i\in J^c$ $\rightarrow$  $h_i= \max \{h:h\,p_i\leq \gamma_i^{max}\}=\frac{\gamma_i^{max}}{p_i},$
\STATE $\forall j\in\{n+1\ldots,n+m\}$ $\rightarrow$  $h_j= \sup \{h:\sum_{i\in J} a_{ji} Q_i +  h(\sum_{i\in J^c} a_{ji}p_i)\leq \gamma_j^{max}\}$.
\STATE Set $\hbar =\min_{ij}\{h_i,h_j\}$.
\IF {$\exists\ j $ s.t. $h_j=\hbar$}
	\STATE Set $I=\{i\in J^c: a_{ji}\not=0\}$ and $Q_i=\hbar\,p_i$  for $i\in I$.
\ELSE 
\STATE Set $I=\{i\in J^c:h_i=\hbar\}$ and $Q_i=\hbar\,p_i$  for $i\in I$.
\ENDIF
\STATE Set $J=J\cup I$.
\ENDWHILE
\caption{Recursive definition of $\SPRS$}
\label{alg:SPRS}
\end{algorithmic}
\end{algorithm}

We are now ready to define the Softer Priority Riemann Solver.
\begin{definition}
Let $Q=(\bar\gamma_1,\ldots,\bar\gamma_n)$ be the vector of incoming fluxes defined by 
Algorithm \ref{alg:SPRS}, then the vector of outgoing fluxes is 
given by $A\cdot Q^T=(\bar\gamma_{n+1,\ldots,n+m})^T$.\\
For every $i\in\{1,\ldots,n\}$, set $\bar\rho_i$ equal either to $\rho_{i,0}$ if $f(\rho_{i,0})=\bar\gamma_i$, or to the solution to $f(\rho)=\bar\gamma_i$ such that $\bar\rho_i\ge\rhocr$.
For every $j\in\{n+1,\ldots,n+m\}$, set $\bar\rho_j$ equal either to $\rho_{j,0}$ if $f(\rho_{j,0})=\bar\gamma_j$, or to the solution
to $f(\rho)=\bar\gamma_j$ such that $\bar\rho_j\le\rhocr$. Finally, $\SPRS:[0,1]^{n+m}\to[0,1]^{n+m}$ is given by
  \begin{equation}\label{eq:rs2_rho}
    \SPRS(\rho_{1,0},\ldots,\rho_{n+m,0})
    =(\bar\rho_1,\ldots,\bar\rho_n,\bar\rho_{n+1},\ldots,\bar\rho_{n+m})\,.
  \end{equation}
\end{definition}
\section{Numerical scheme and numerical simulations}\label{sec:num_sim}
%
To illustrate the $\PRS$ and $\SPRS$ dynamics we provide some simulations
based on the well-known Godunov scheme \cite{G59} on networks (see \cite{GPbook}),
which is based on solutions to Riemann problems.\\
Define a numerical grid on $[0,T] \times \R$ given by:
\begin{itemize}
\item $\dx$ is the fixed space grid size;
\item $\dt^\nu$, $\nu\in\N$, is the  time grid size satisfying the CFL condition \cite{CFL67}:
\begin{equation}
\label{eq:CFL}
\dt^\nu \max_{j\in\Z} \abs{f'(u_j^\nu)}\leq \dfrac{1}{2}\dx
\end{equation}
\item $(t^\nu, x_j) = (t^{\nu-1}+\dt^\nu, j\dx)$ for $\nu\in \N$ and $j\in \Z$ are the grid points.
\end{itemize}
Consider a scalar conservation laws equipped with initial data:
\begin{equation}
\label{eq:SCL_GOD}
\begin{array}{ll}
\partial_t u + \partial_x f(u) = 0, \qquad &x\in \R,\ t\in [0,T],\\
u(0,x) = u_0(x) ,\qquad &x\in \R .
\end{array}
\end{equation}
An approximate solution of the problem is constructed first by taking a piecewise constant approximation of the initial data
\begin{equation}
\label{eq:PW_ID_GOD}
u^{0}_j = \dfrac{1}{\dx}\int^{\xj{+}{1}}_{\xj{-}{1}}u_0(x)dx,  \qquad j\in \Z,
\end{equation}
and then defining $u^\nu_j$ recursively from $u^{0}_j$ as follows.
Under the CFL ~\eqref{eq:CFL}
the waves generated by different Riemann problem at the cell interfaces do not interact
and the scheme can be written as follows
\begin{equation}
\label{eq:god_sheme}
u_j^{\nu+1}=u_j^{\nu}-\dfrac{\dt^\nu}{\dx}\left(G(u_j^{\nu},u^{\nu}_{j+1}) - G(u^\nu_{j-1},u_j^\nu)\right),
\end{equation}
where the numerical flux $G$ is given by
\begin{equation}
\label{eq:god_flux}
G(u,v) = \left\lbrace
\begin{array}{ll}
\min_{z\in [u,v]} f(z) \qquad &\text{if } u\leq v  \\
\max_{z\in [v,u]} f(z) \qquad &\text{if } v\leq u.
\end{array}
\right.
\end{equation}
To impose boundary conditions and conditions at junctions we use the classical approach introduced in \cite{BNP06}.\\
\textbf{Boundary conditions}.\\
Each road is divided into $M$ cells, numbered from $1$ to $M$. Boundary conditions are imposed using ghost cells. For an incoming road $I_i$ we define:
\begin{equation}
u_{i,1}^{\nu+1} = u_{i,1}^\nu -\dfrac{\dt^\nu}{\dx}\left(G(u_{i,1}^{\nu},u^{\nu}_{i,2}) - G(u^\nu_{i,0},u_{i,1}^\nu)\right)
\end{equation}
where $u^\nu_{i,0}$ is the value of the density at the boundary.\\
The outgoing boundary for $I_j$ is treated in the same way by defining:
\begin{equation}
u_{j,M}^{\nu+1} = u_{j,M}^\nu -\dfrac{\dt^\nu}{\dx}\left(G(u_{j,M}^{\nu},u^{\nu}_{j,M+1}) - G(u^\nu_{j,M-1},u_{j,M}^\nu)\right)
\end{equation}
with $u^\nu_{i,M+1}$ the value of the density at the outgoing boundary.\\
\textbf{Conditions at the junction}.\\
For $I_i$ with $i \in \{1,\ldots n\}$ that is connected at the junction at the right endpoint we set:
\begin{equation}
u_{i,M}^{\nu+1}=u_{i,M}^\nu - \dfrac{\dt^\nu}{\dx}\left(Q_i^\nu - G(u^\nu_{i,M-1},u_{i,M}^\nu)\right) \qquad \text{for } i \in \{1,\ldots n\},
\end{equation}
while for the outgoing roads, connected at the junction with the left endpoint we have:
\begin{equation}
u_{j,1}^{\nu+1}=u_{j,1}^\nu - \dfrac{\dt^\nu}{\dx}\left(G(u_{j,1}^{\nu},u^{\nu}_{j,2}) - Q_j^\nu \right) \qquad \text{for } j \in \{n+1,\ldots n+m\},
\end{equation}
where $Q_i^\nu, Q_j^\nu$ are the incoming and outgoing fluxes
given by the Riemann solvers at junction corresponding to the initial data 
\[
(u_{1,M}^\nu,\ldots,u_{n,M}^\nu,u_{n+1,1}^\nu,\ldots,u_{n+m,1}^\nu)
\]
(see Algorithms \ref{alg:PRS}, \ref{alg:SPRS}).

\subsection{Numerical results}
For the simulations, we set the length of each road equal to $1$ and incoming roads are parametrized by the interval $I_i = [-1,\ 0]$ while outgoing roads are given by $I_j = [0,\ 1]$,  with the junction placed at $x=0$. Moreover, we fix $f(\rho)=\rho(1-\rho)$, thus $\rhocr=0.5.$
\begin{enumerate}
\item \textbf{Case I:} Comparison $\PRS$ vs. $\SPRS$. \\
This case illustrates the different dynamics given by the two Riemann solvers proposed in this article.\\ 
We consider a junction with $2$ incoming roads ($I_1,\ I_2$) and $2$ outgoing roads ($I_3, I_4$).
We fix the matrix $A$ and the priority vector $P$ as follows:
\begin{equation}
\label{eq:CaseI_MatA}
A= \begin{bmatrix}
0.6 & 0 \\
0.4 & 1\\
\end{bmatrix} \quad
P = \begin{bmatrix}
0.7 & 0.3
\end{bmatrix}.
\end{equation}
We consider the following initial data:
\begin{equation}
\label{eq:CaseI_InData}
\rho_{1,0} = 0.6, \qquad \rho_{2,0} = 0.2, \qquad \rho_{3,0} = 0.85, \qquad \rho_{4,0} = 0.2.
\end{equation}
\begin{figure}[!ht]
	\begin{subfigure}[H]{\textwidth}
		\centering
		\includegraphics[width=0.31\textwidth]{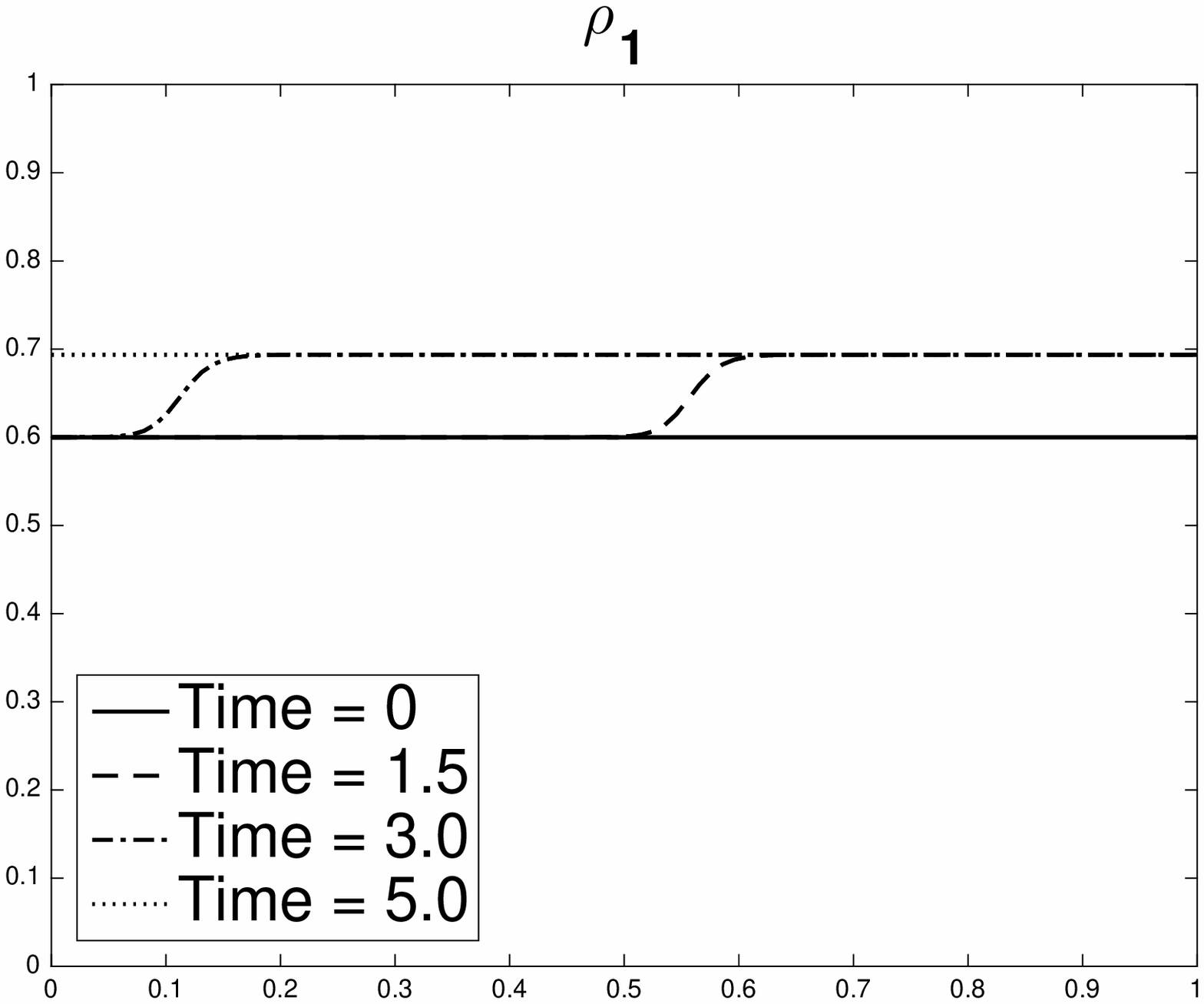} 
		\hskip 1cm
		\includegraphics[width=0.31\textwidth]{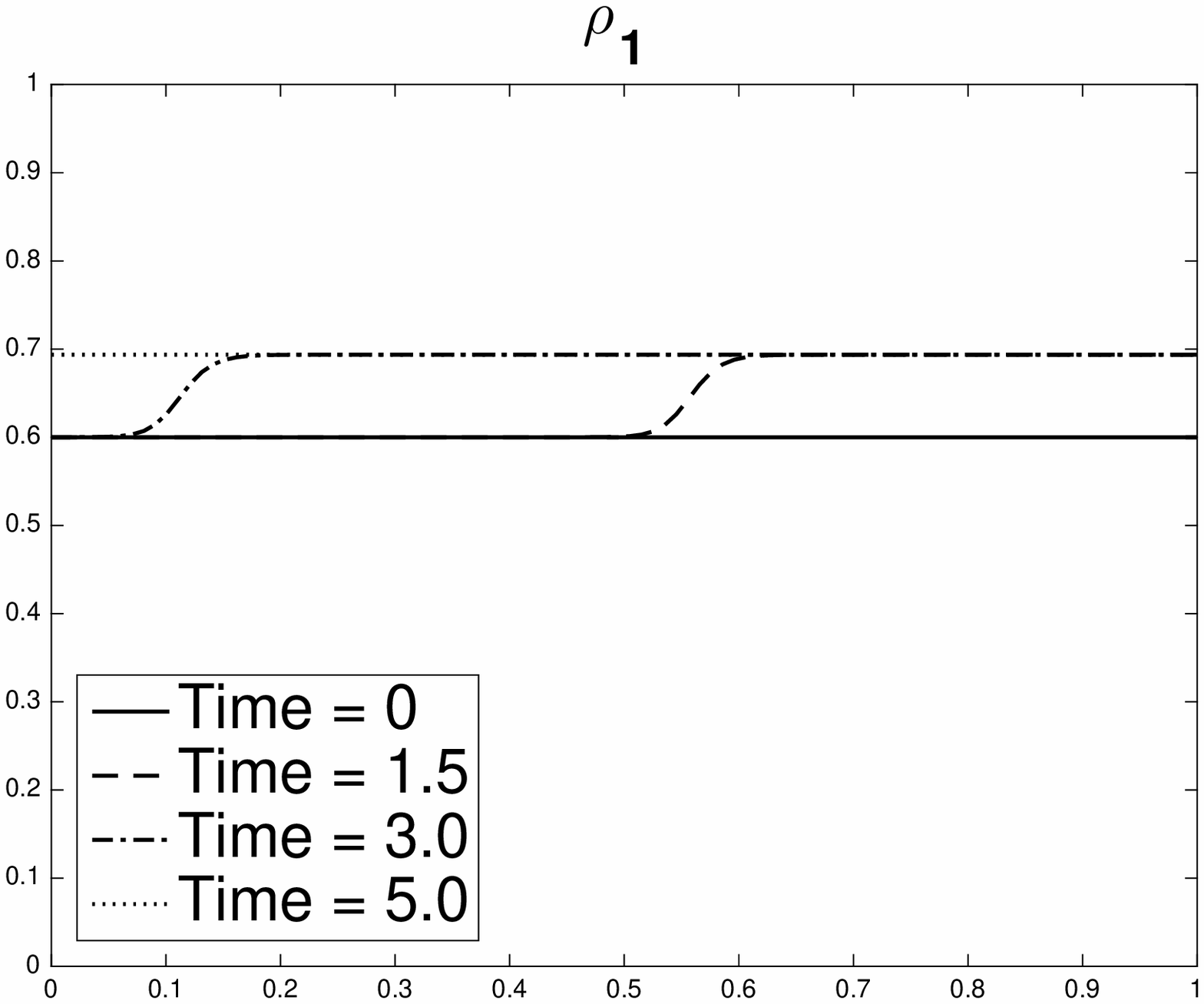}
		\caption{Road 1}
	\end{subfigure}\\
	\begin{subfigure}[H]{\textwidth}
		\centering
		\includegraphics[width=0.31\textwidth]{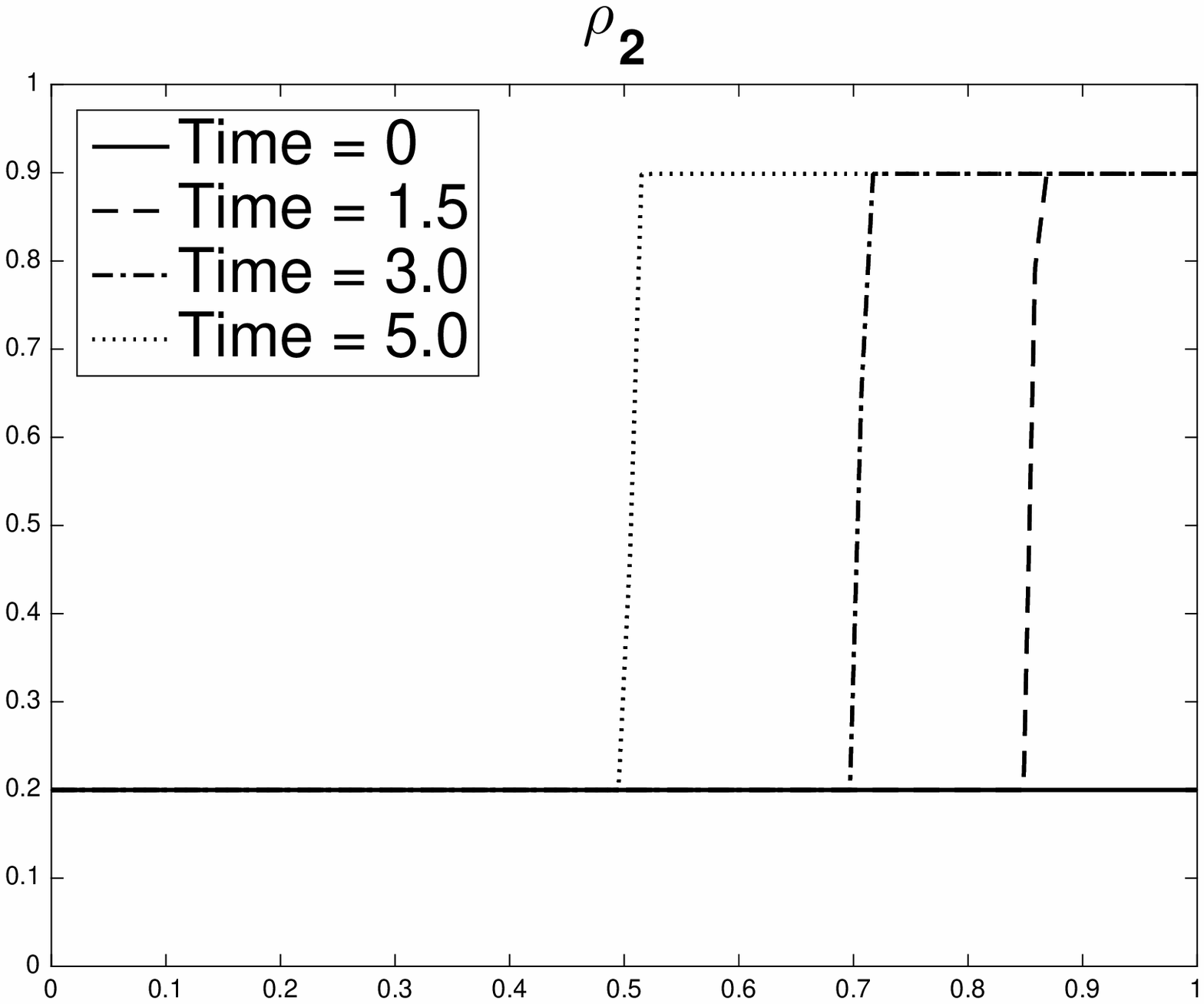}
		\hskip 1cm
		\includegraphics[width=0.31\textwidth]{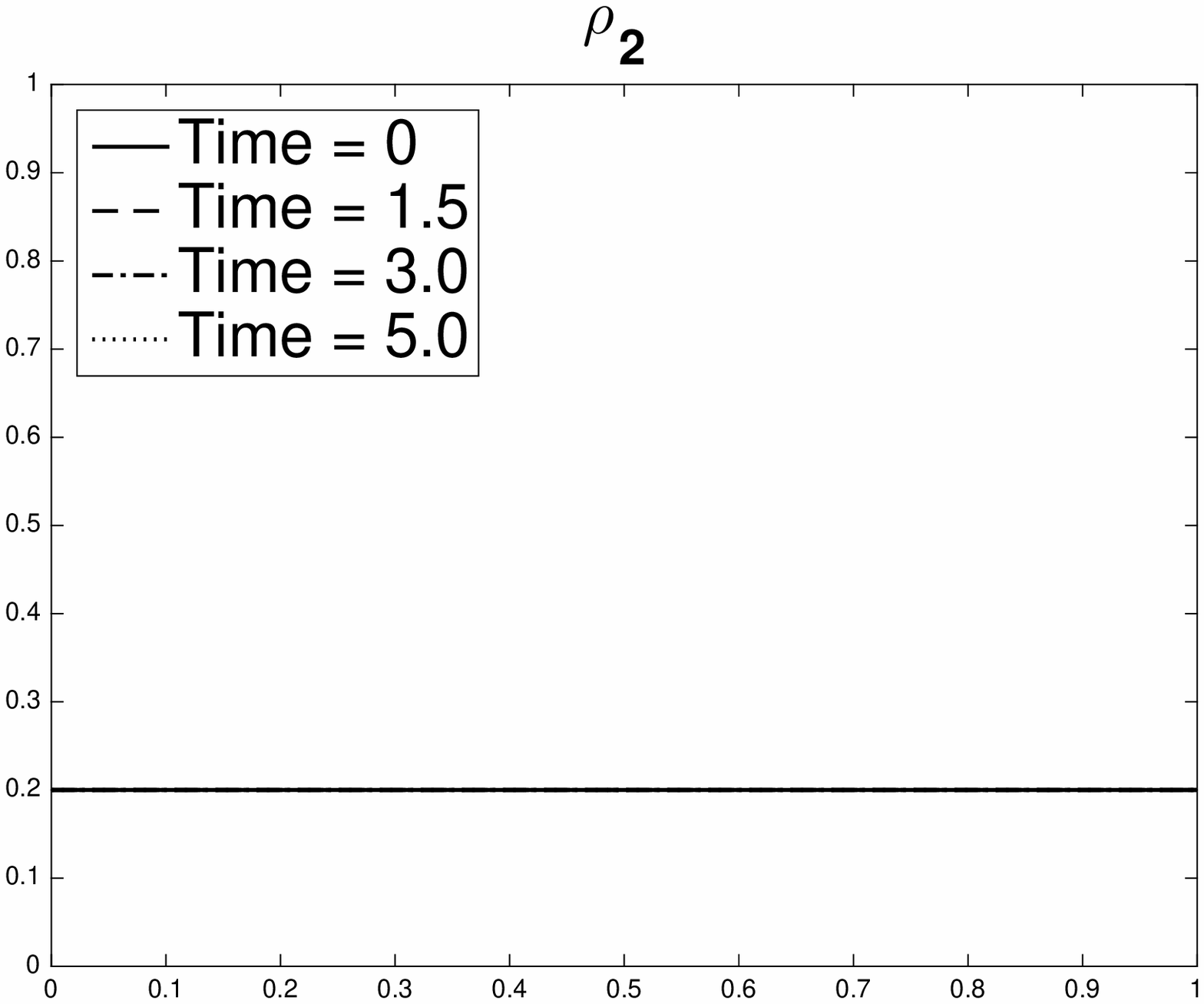}
		\caption{Road 2}
	\end{subfigure}
	\begin{subfigure}[H]{\textwidth}
		\centering
		\includegraphics[width=0.31\textwidth]{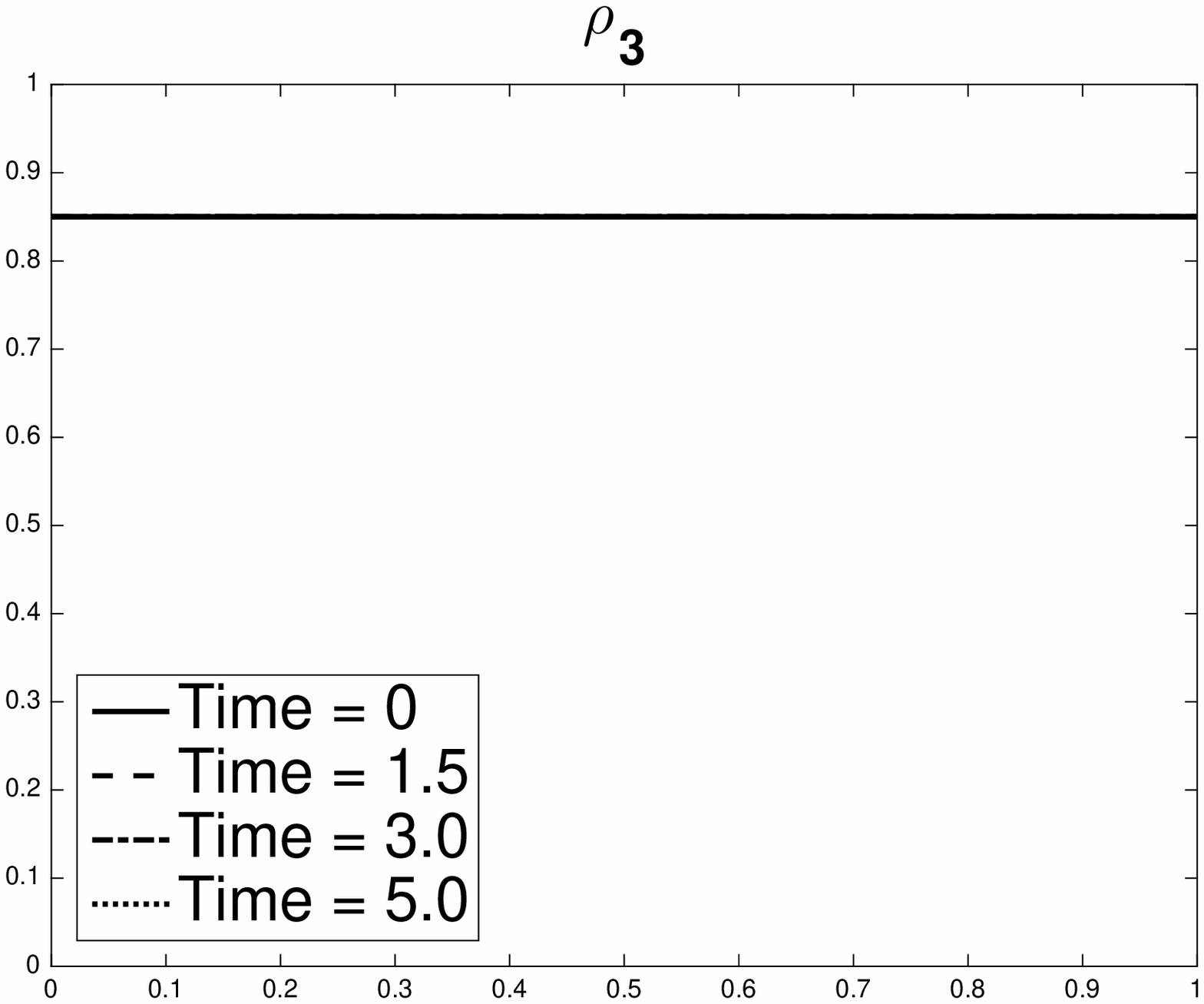}
		\hskip 1cm
		\includegraphics[width=0.31\textwidth]{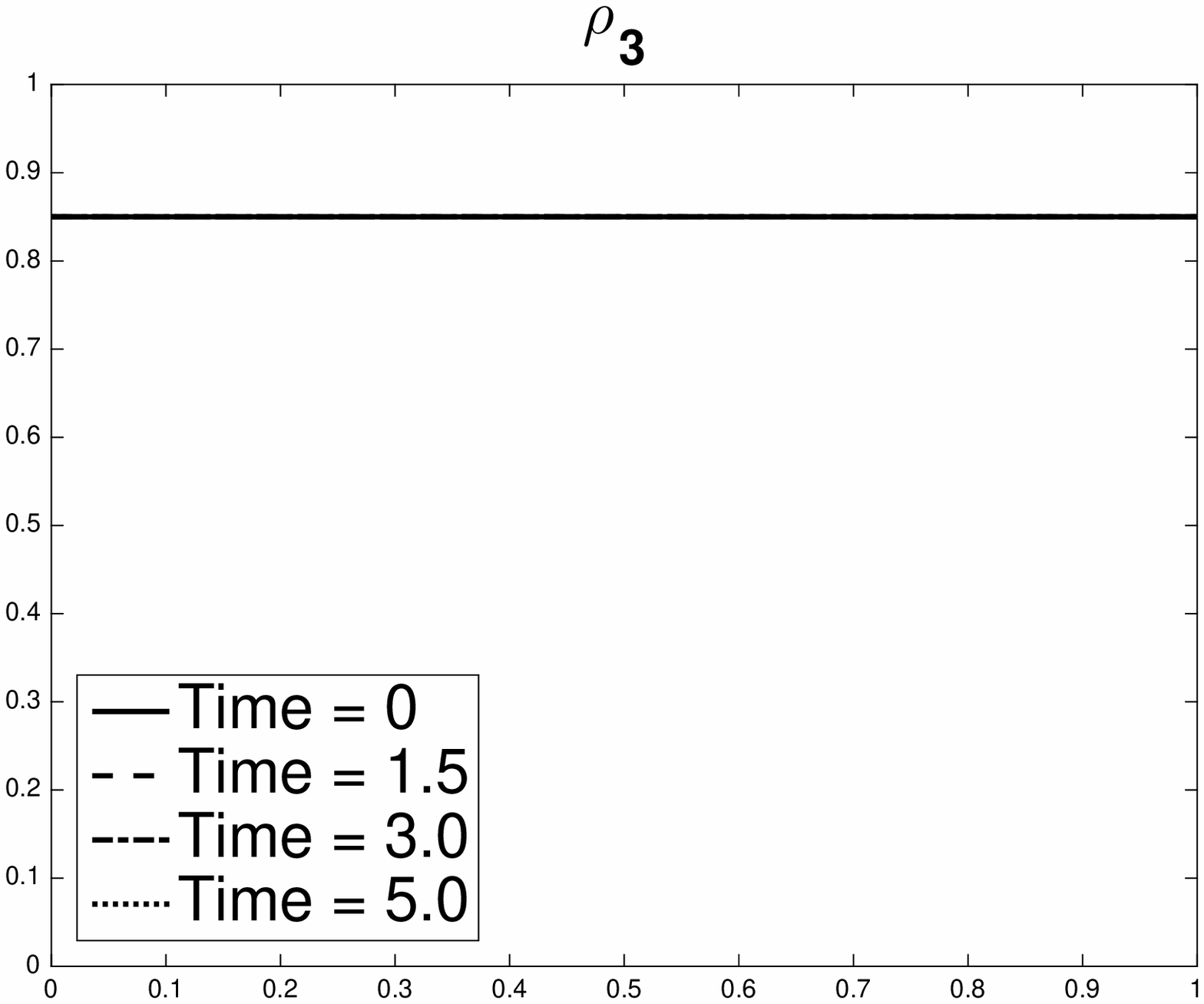}
		\caption{Road 3}
	\end{subfigure}
	\begin{subfigure}[H]{\textwidth}
		\centering
		\includegraphics[width=0.31\textwidth]{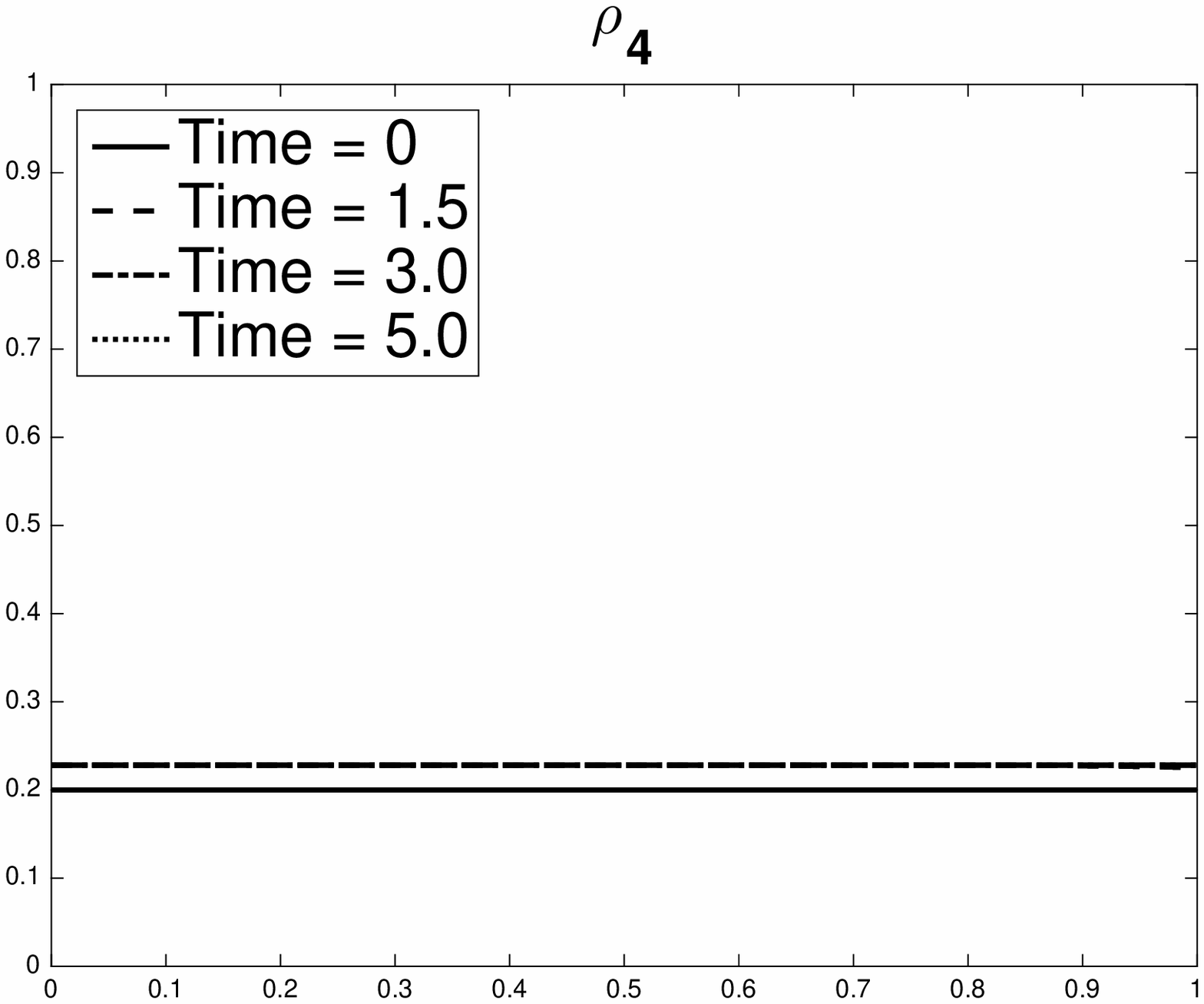}
		\hskip 1cm
		\includegraphics[width=0.31\textwidth]{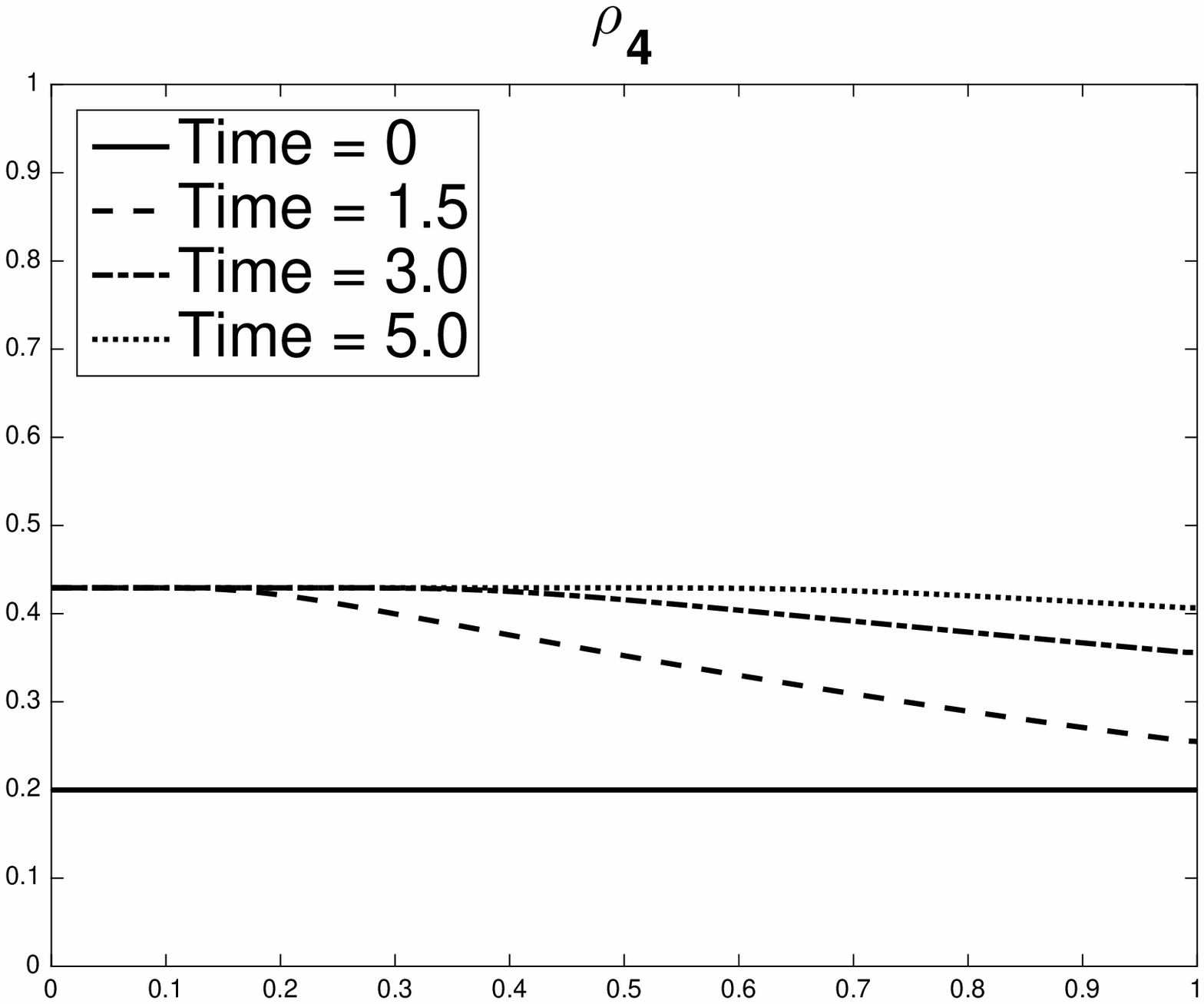}
		\caption{Road 4}
	\end{subfigure}
\caption{Case I : Solution of the problem using $\PRS$ on the left and $\SPRS$ on the right.}
\label{fig:caseI}
\end{figure}
The different results of the simulations (see Figures \ref{fig:caseI}) can be seen in particular in road $2$ ad $4$. We observe that $\SPRS$ allows more flux through the junction than $\PRS$, for which we observe the formation of a big shock
moving backwards on road 2.

\item \textbf{Case II:} Comparison $\PRS$ vs. $\Rsol_{\mathrm{CGP}}$. \\
We propose here a comparison between the $\PRS$ with the Riemann solver proposed by Coclite, Garavello and Piccoli in \cite{CGP05} and briefly referred to as $\Rsol_{\mathrm{CGP}}$. \\ 
We consider a $2\times 2$ junction and 
we fix the matrix $A$ and the priority vector $P$ as follows:
\begin{equation}
\label{eq:CaseII_MatA}
A= \begin{bmatrix}
0.5 & 0.6 \\
0.5 & 0.4\\
\end{bmatrix} \quad
P = \begin{bmatrix}
0.7 & 0.3
\end{bmatrix}.
\end{equation}
We consider the following initial data:
\begin{equation}
\label{eq:CaseII_InData}
\rho_{1,0} = 0.2, \qquad \rho_{2,0} = 0.6, \qquad \rho_{3,0} = 0.3, \qquad \rho_{4,0} = 0.8.
\end{equation}
\begin{figure}[!ht]
	\begin{subfigure}[H]{\textwidth}
		\centering
		\includegraphics[width=0.31\textwidth]{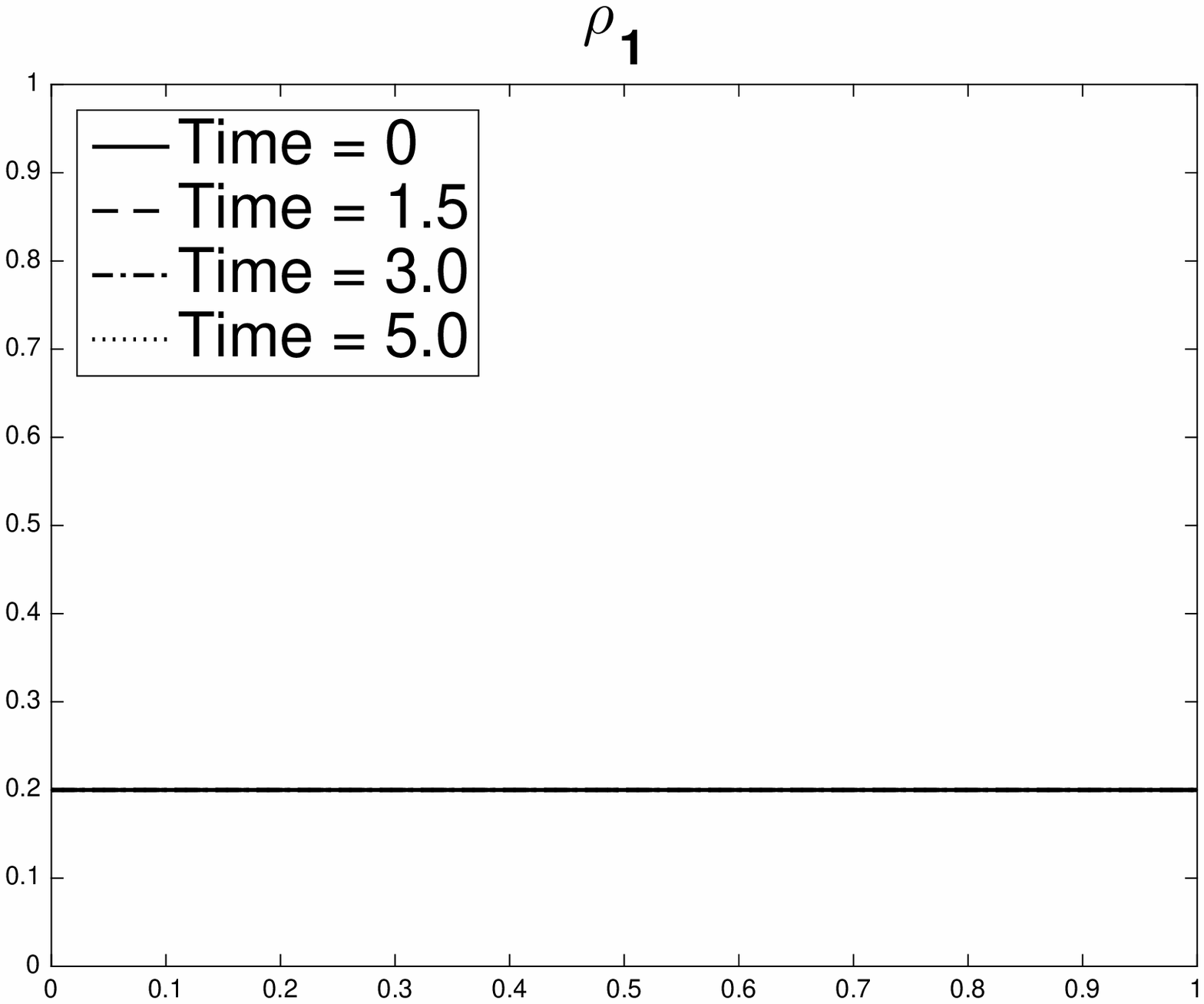}
		\hskip 1cm
		\includegraphics[width=0.31\textwidth]{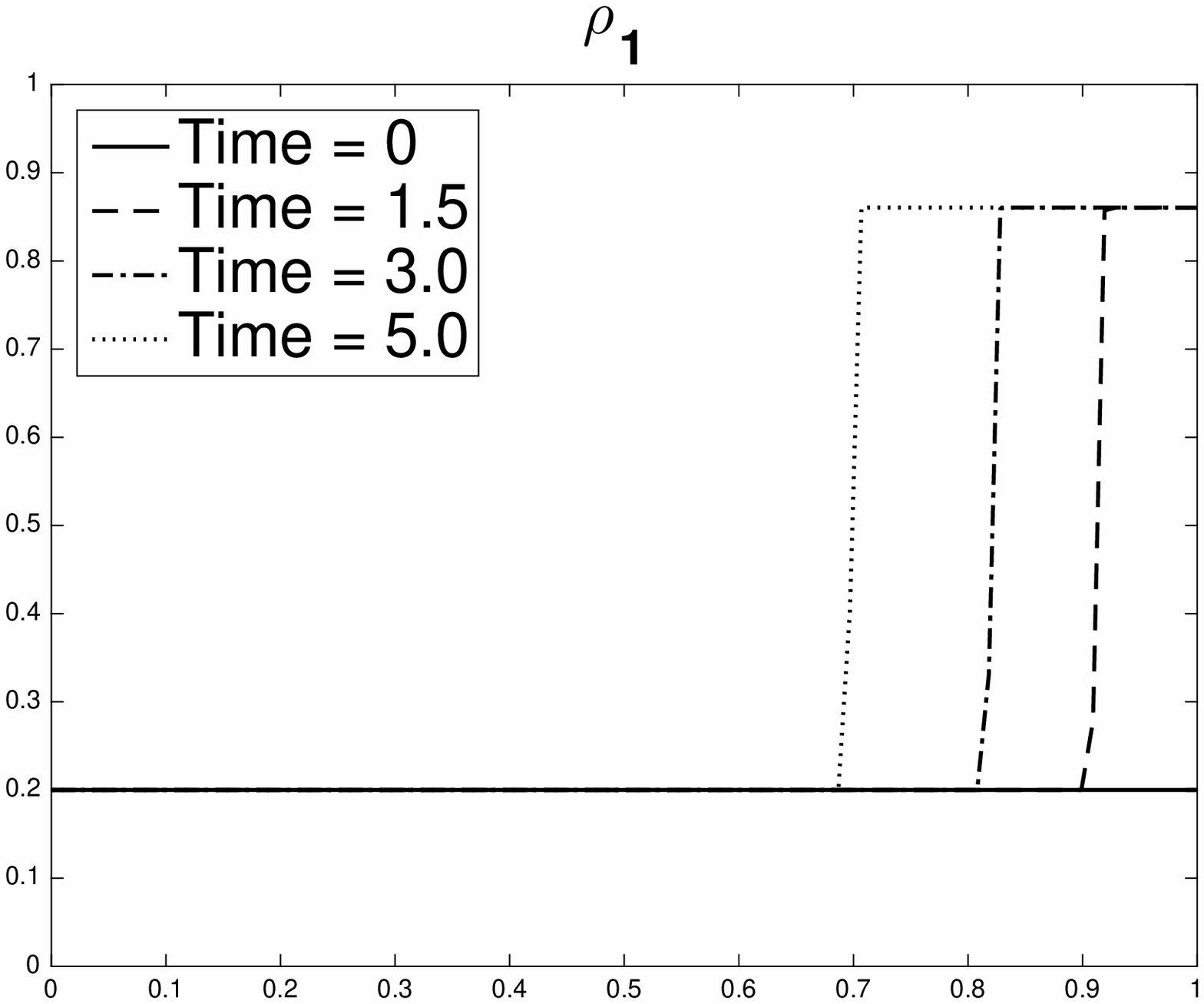}
		\caption{Road 1}
	\end{subfigure}\\
	\begin{subfigure}[H]{\textwidth}
		\centering
		\includegraphics[width=0.31\textwidth]{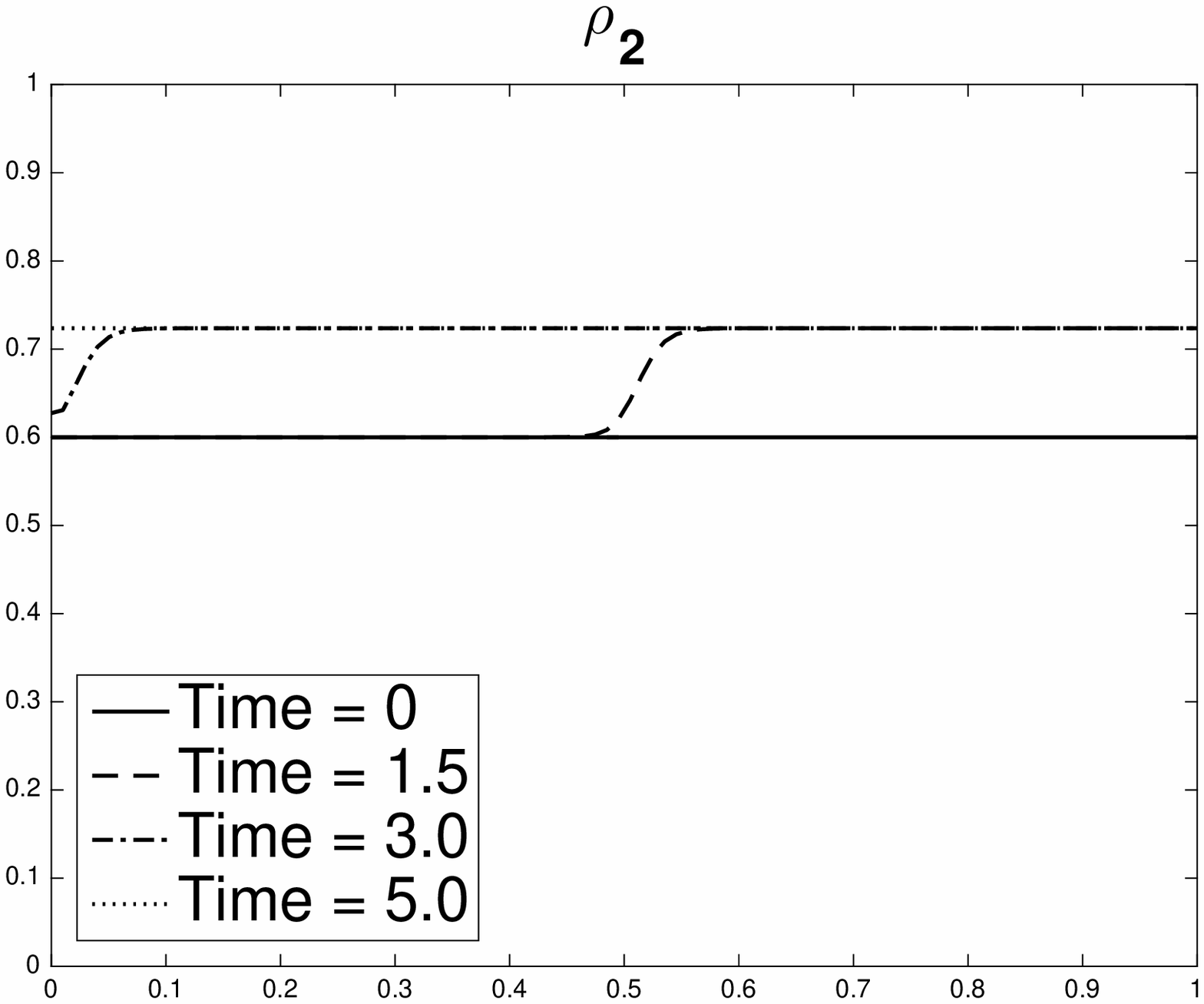}
		\hskip 1cm
		\includegraphics[width=0.31\textwidth]{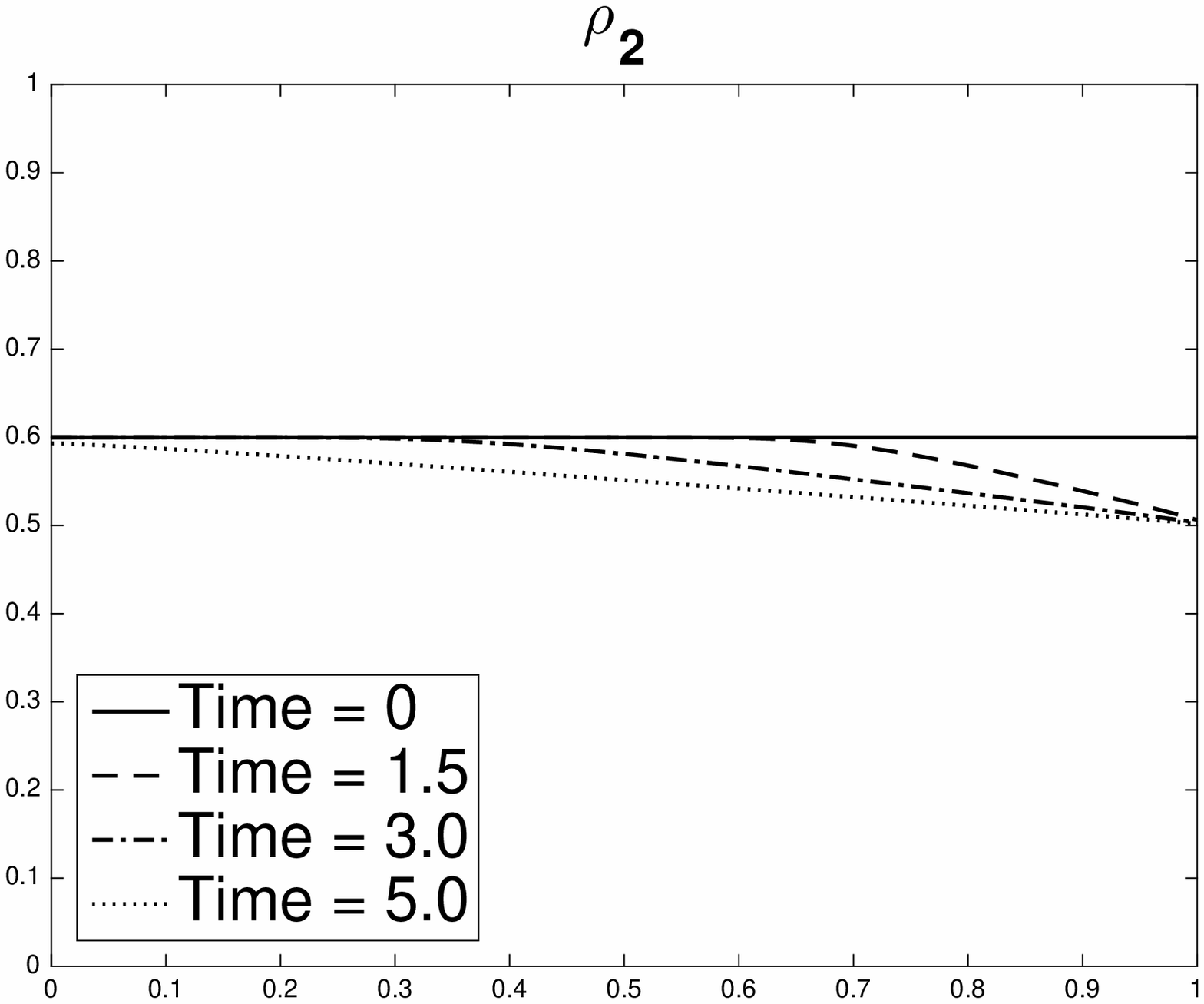}
		\caption{Road 2}
	\end{subfigure}
	\begin{subfigure}[H]{\textwidth}
		\centering
		\includegraphics[width=0.31\textwidth]{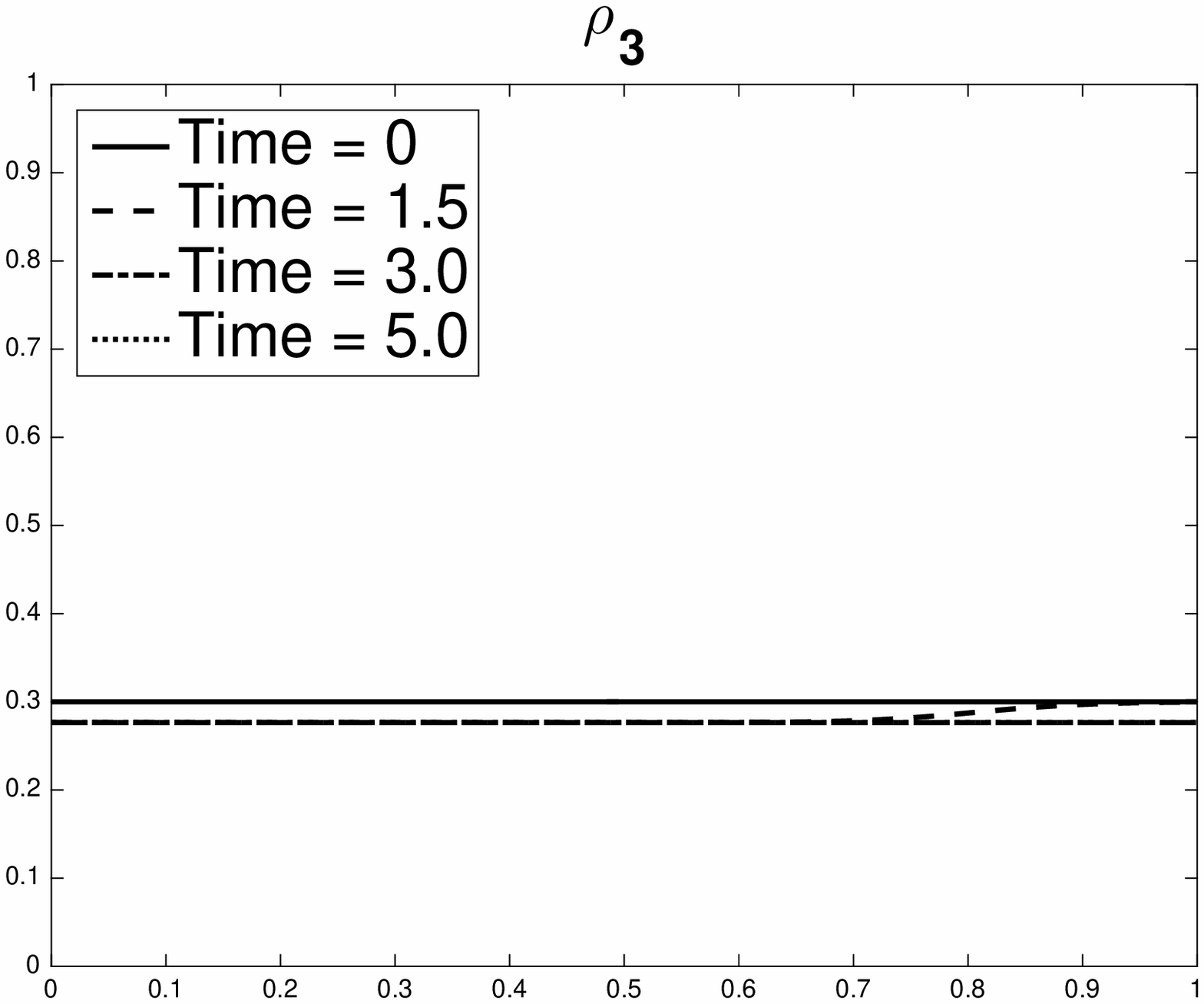}
		\hskip 1cm
		\includegraphics[width=0.31\textwidth]{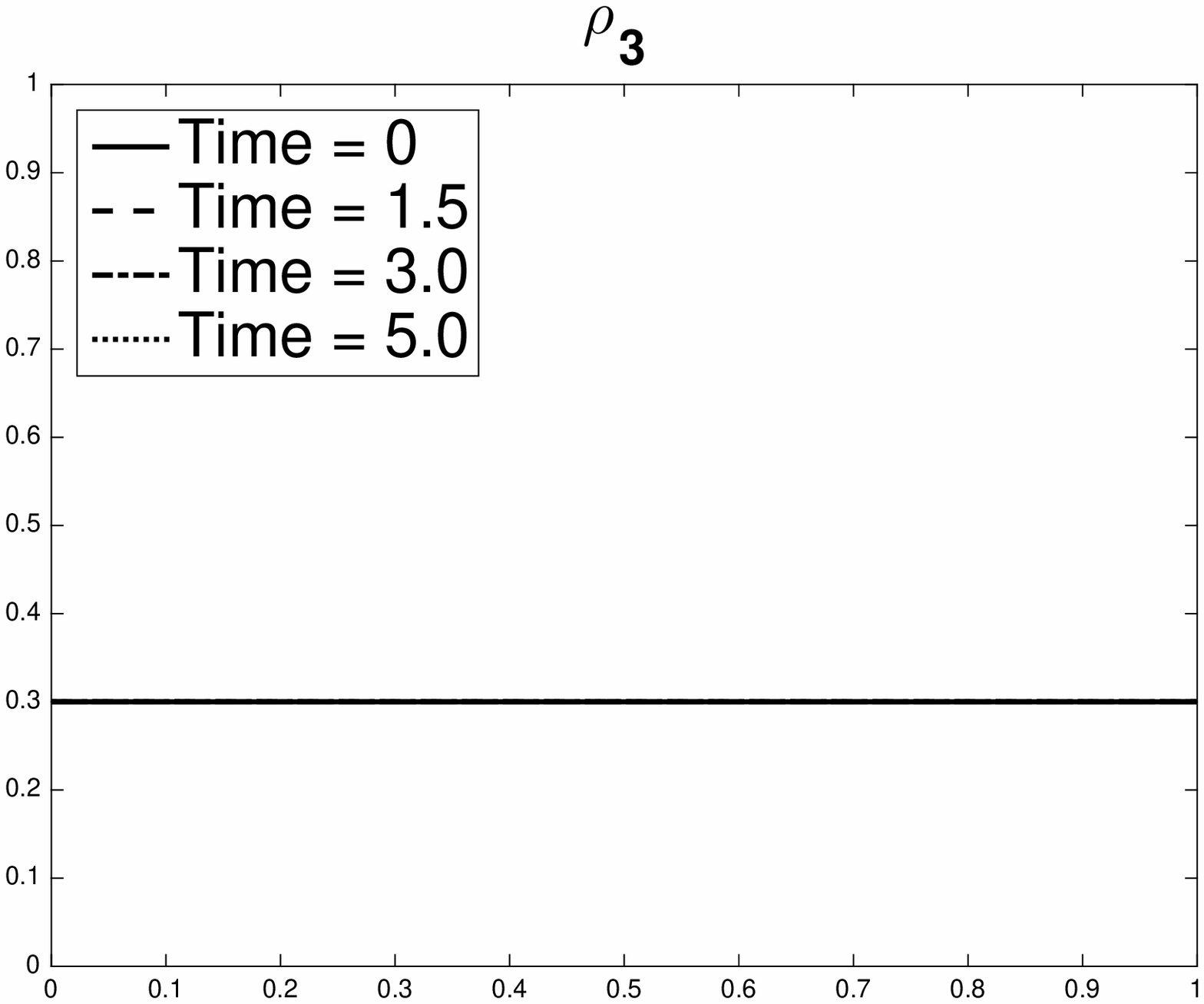}
		\caption{Road 3}
	\end{subfigure}
	\begin{subfigure}[H]{\textwidth}
		\centering
		\includegraphics[width=0.31\textwidth]{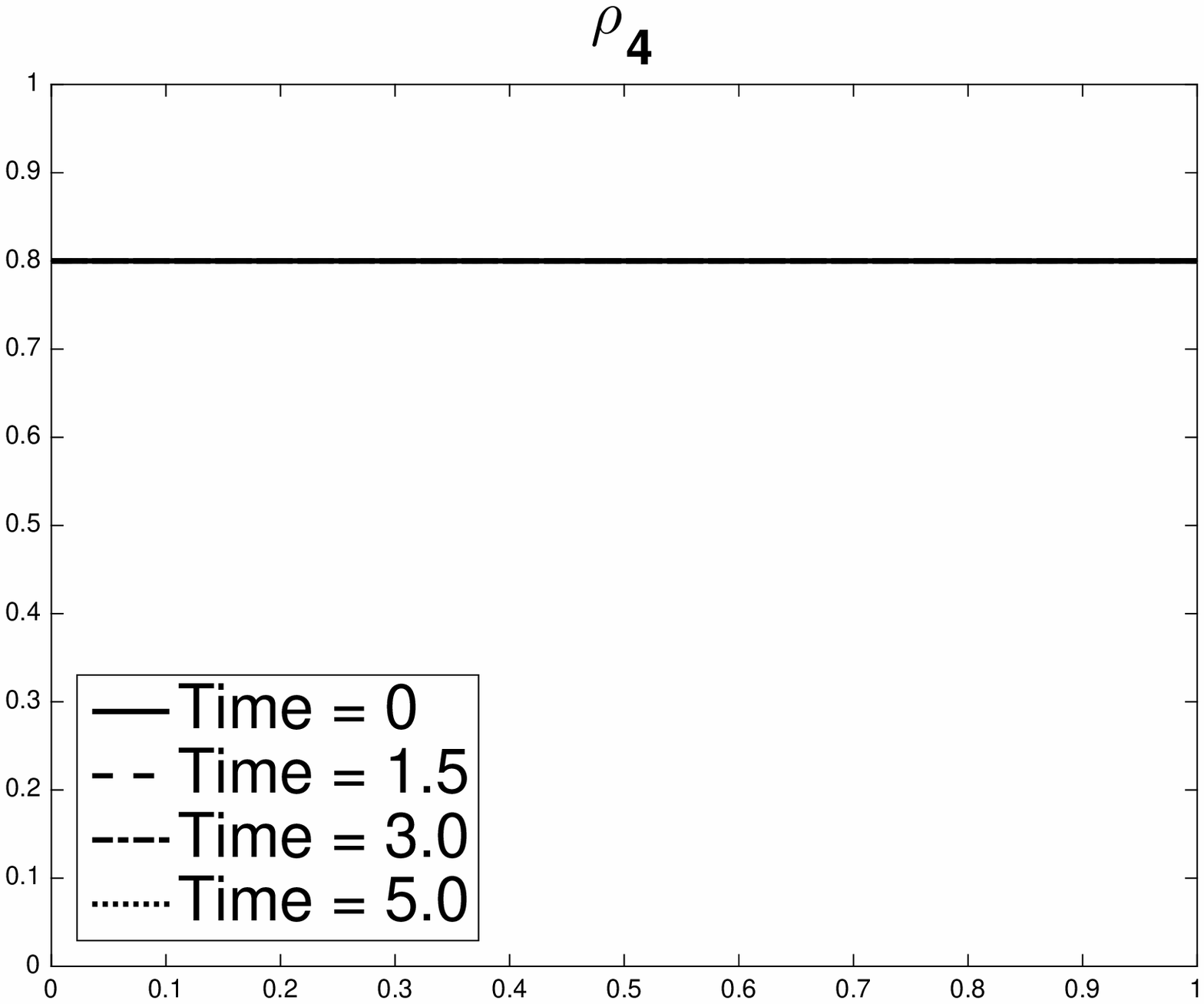}
		\hskip 1cm
		\includegraphics[width=0.31\textwidth]{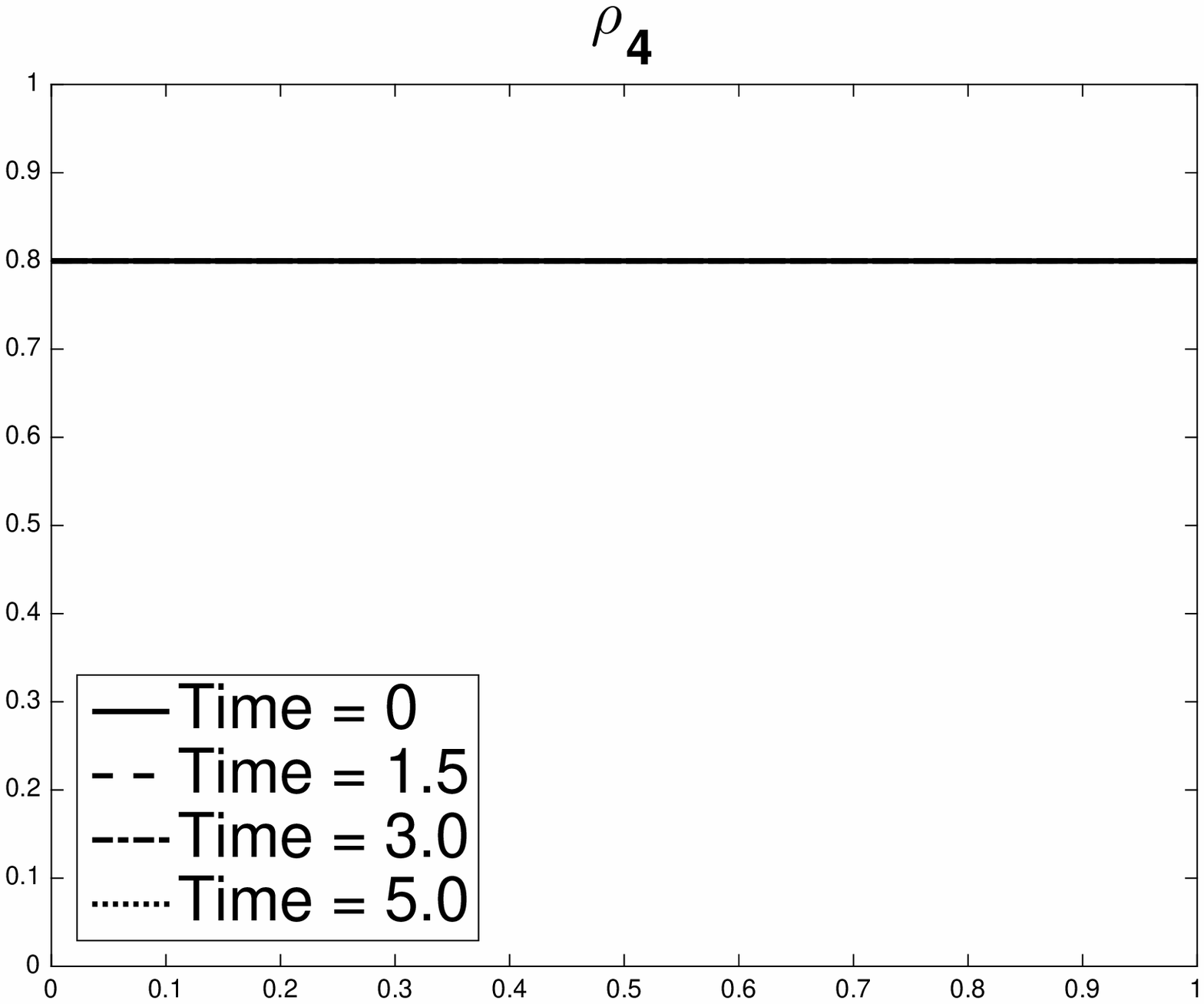}
		\caption{Road 4}
	\end{subfigure}
\caption{Case II : Solution of the problem using $\PRS$ on the left and $\Rsol_{\mathrm{CGP}}$ on the right.}
\label{fig:caseII}
\end{figure}

The simulations (see Figures \ref{fig:caseII}) show clearly the different solutions of the Riemann solvers. In particular,  $\Rsol_{\mathrm{CGP}}$ creates a big shock in the incoming road $1$ decreasing  its flux. This wave does not appear in our Riemann solver $\PRS$.
\FloatBarrier
\item \textbf{Case III:} $3\times 2$ junction.

We fix the matrix $A$ and the priority vector $P$ as follows:
\begin{equation}
\label{eq:CaseIII_MatA}
A= \begin{bmatrix}
0.5 & 0.6 & 0.2\\
0.5 & 0.4 & 0.8\\
\end{bmatrix} \quad
P = \begin{bmatrix}
0.5 & 0.3 & 0.2
\end{bmatrix}.
\end{equation}
We consider the following initial data:
\begin{equation}
\label{eq:CaseIII_InData}
\rho_{1,0} = 0.2, \qquad \rho_{2,0} = 0.6, \qquad \rho_{3,0} = 0.3, \qquad \rho_{4,0} = 0.8, \qquad \rho_{5,0} = 0.2. 
\end{equation}
\begin{figure}[H]
	\centering
	\begin{subfigure}[H]{0.3\textwidth}
		\centering
		\includegraphics[width=\textwidth]{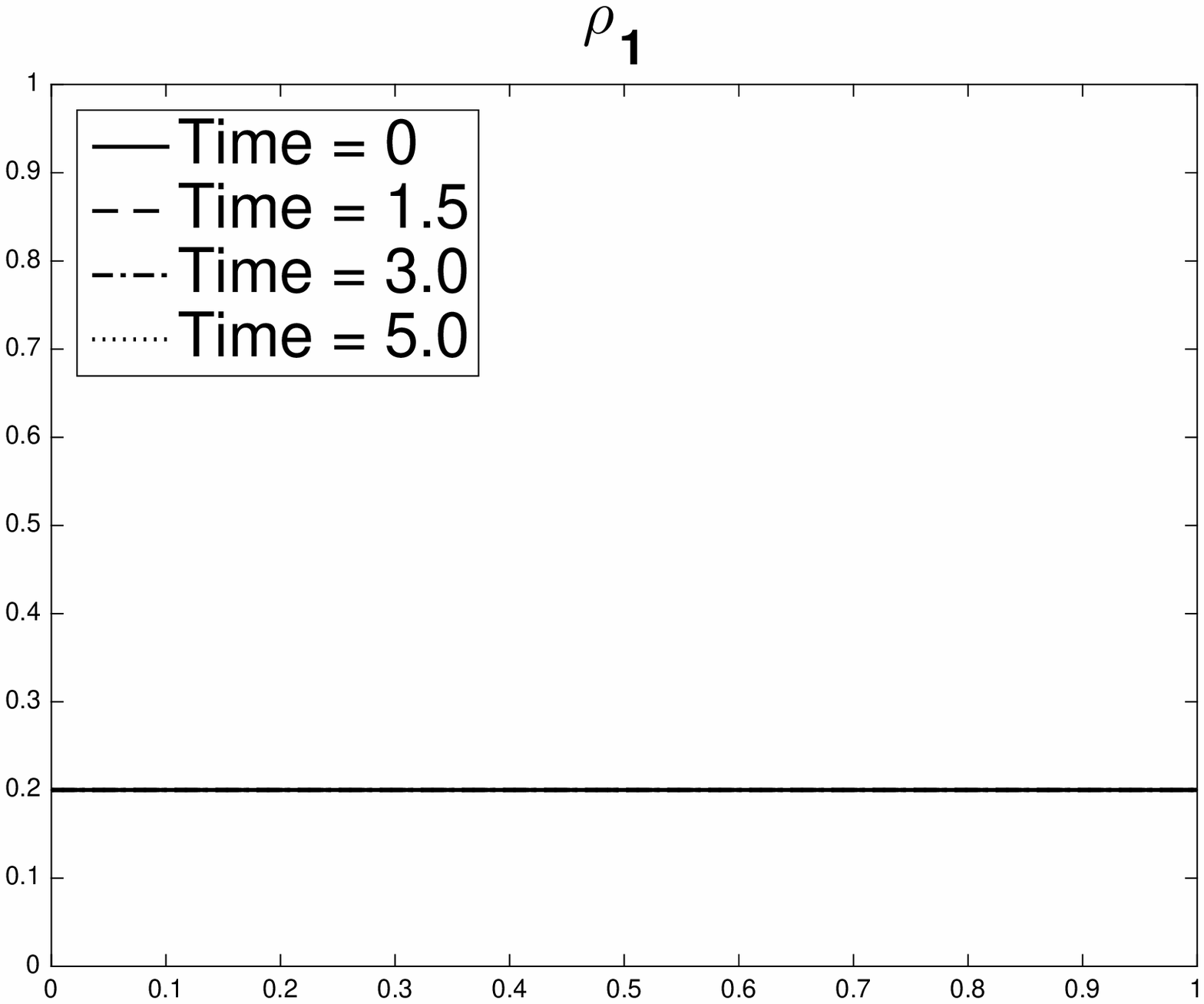}
		\caption{Road 1}
	\end{subfigure}
	\hskip 1cm
	\begin{subfigure}[H]{0.3\textwidth}
		\centering
		\includegraphics[width=\textwidth]{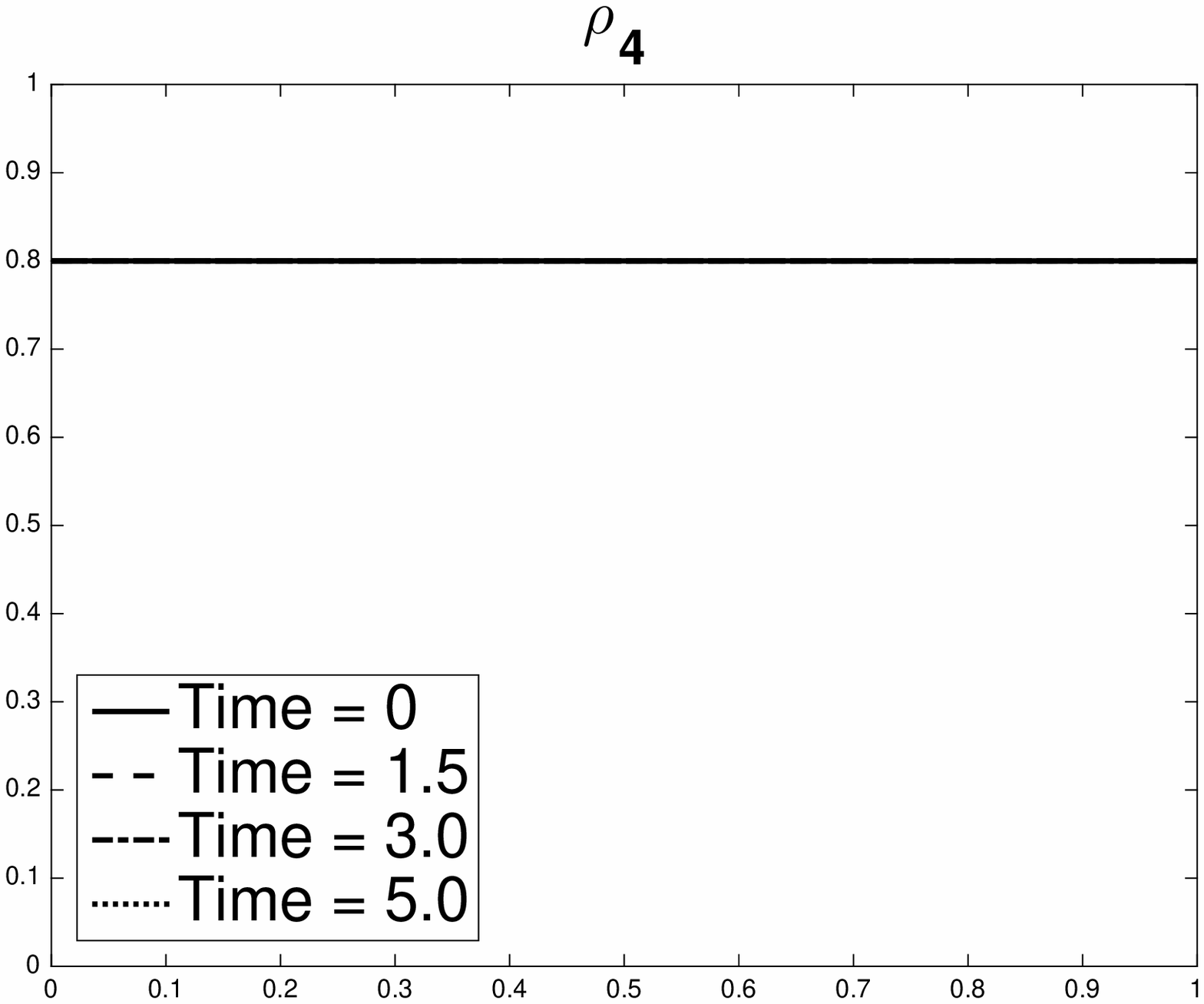}
		\caption{Road 4}
	\end{subfigure}\\
	\begin{subfigure}[H]{0.3\textwidth}
		\centering
		\includegraphics[width=\textwidth]{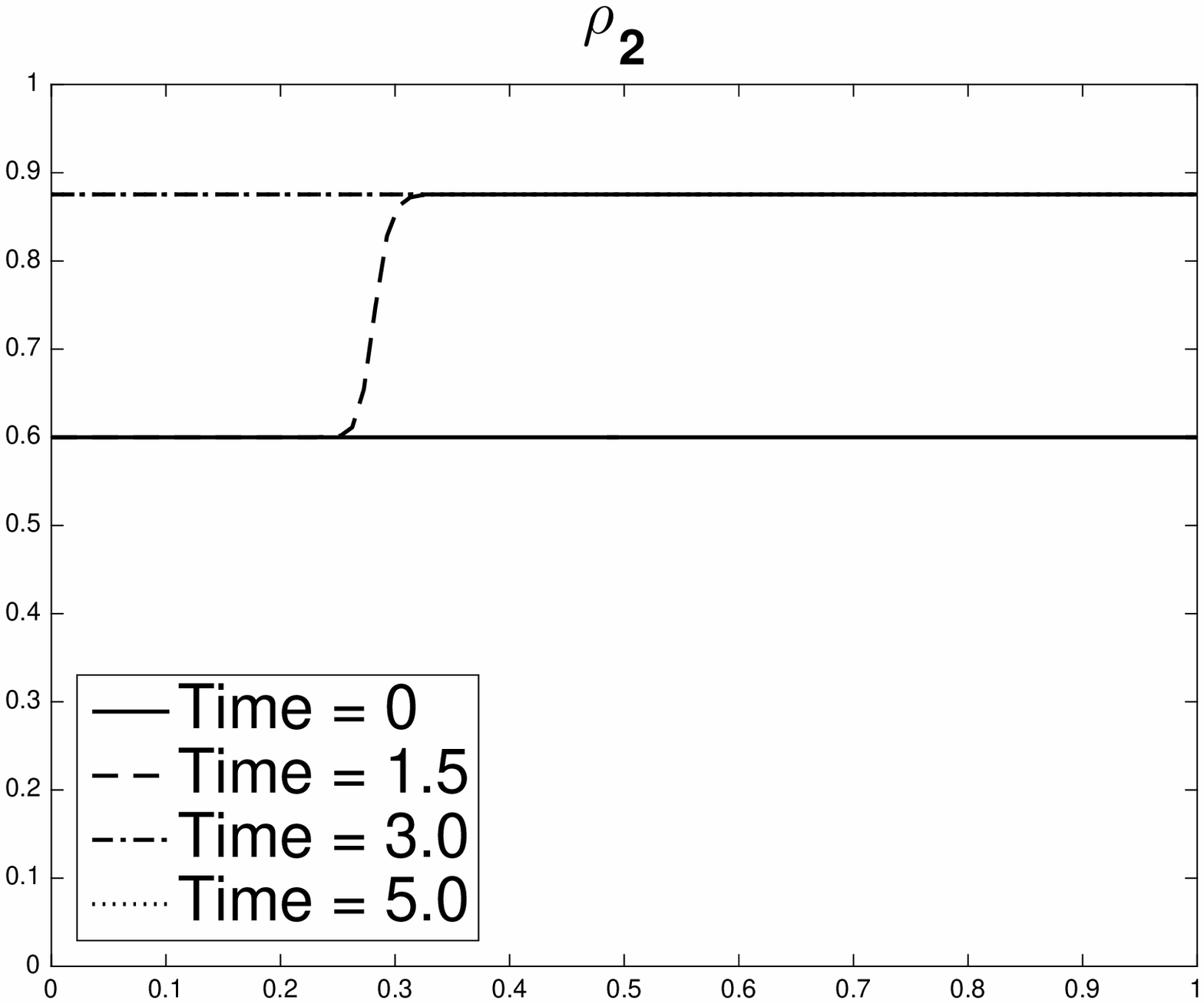}
		\caption{Road 2}
	\end{subfigure}
	\hskip 1cm
	\begin{subfigure}[H]{0.3\textwidth}
		\centering
		\includegraphics[width=\textwidth]{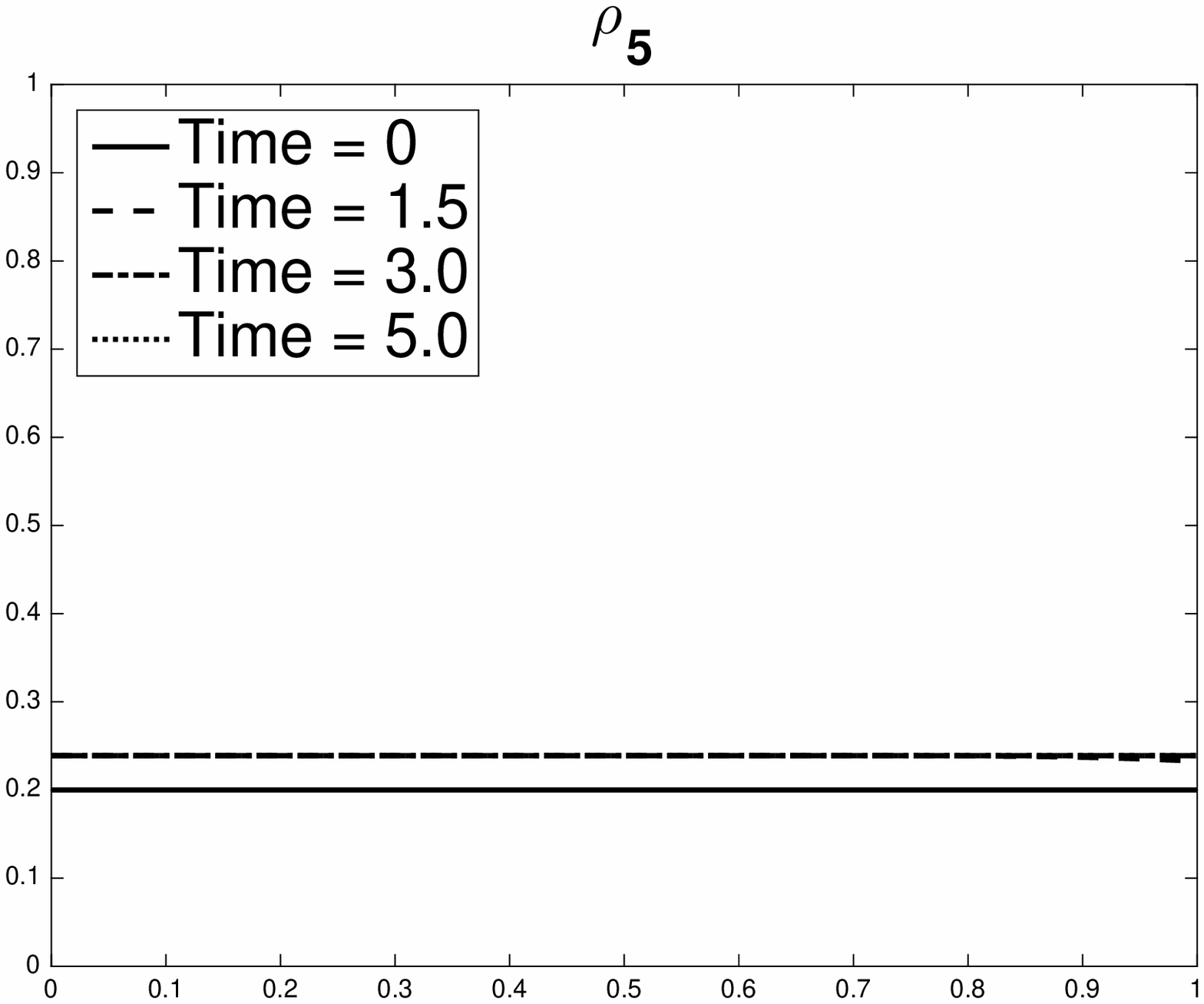}
		\caption{Road 5}
	\end{subfigure}\\
	\begin{subfigure}[H]{0.3\textwidth}
		\flushleft
		\includegraphics[width=\textwidth]{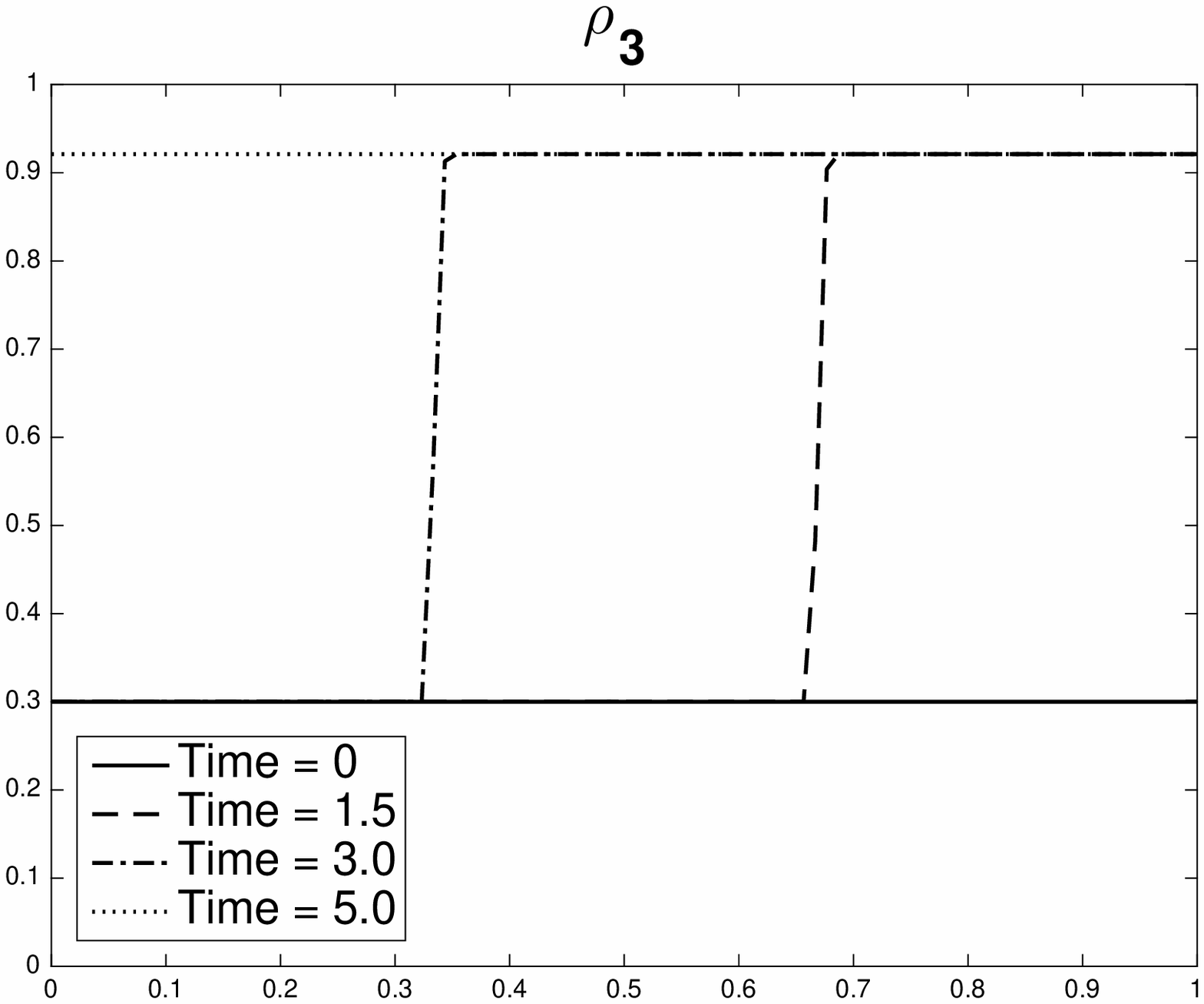}
		\caption{Road 3}
	\end{subfigure}
\caption{Case III : Solution of the problem using $\PRS$}
\label{fig:caseIII}
\end{figure}
Due to the lower priorities given to roads $2$ and $3$ we can see that queues are created in the two incoming roads, see Figure \ref{fig:caseIII}.  Note also that this case cannot be handled by $\Rsol_{\mathrm{CGP}}$ since $n>m$.
\end{enumerate}

\FloatBarrier

\section{Appendix: Proof of Proposition \ref{prop:P123}}
\label{sec:appendix}
The construction of $\PRS$ depends only on the matrix $A$,
the priority vector $P$ and the sets $\Omega_l$.
The latter, in turn, depends only on bad data,
thus property (P1) holds true.

We prove (P2) and (P3) for the case $n=m=2$ 
and distinguish between three different generic situations for the initial equilibrium:
demand constrained, demand/supply constrained and supply constrained
(where demand indicates flow from incoming road and supply flow to outgoing ones).
In the first situation the incoming roads act as constraint in the definition of
the set $\Omega$ (see (\ref{eq:omega})) and the equilibrium corresponds to the point $Q_0$ as
in Figure \ref{fig:CaseA1_increasing}. The second case corresponds to one incoming and one
outgoing road acting as constraint and to the point $Q_0$ as in Figure \ref{fig:CaseB1_increasing}.
Finally, the third case corresponds to outgoing roads acting as constraint and
to the point $Q_0$ as in Figure \ref{fig:CaseC3_increasing}.

\begin{itemize}
\item {\bf Case A: Demand constrained.} By symmetries, it is not restrictive to assume that the priority line $hP$, $h>0$,
intersects the constraint $\gamma_2^\mathrm{inc}=\gamma_{2,0}$. We have to distinguish several subcases: \\[10pt]
{\bf Case A1:} The incoming wave is $(\rho_1,\rho_{1,0})$ (on road 1). Since $\gamma_{1,0}$ is an active constraint,
$\rho_{1,0}\leq \rhocr$ and $\rho_1\leq \rhocr$. 
We distinguish the two situations:\\[10pt]
$-$ If $f(\rho_1)>f(\rho_{1,0})$ we define $\gamma_{1,1}=\dfrac{\gamma_3-a_{32}\gamma_2}{a_{31}}$
and $\gamma_{1,0}\leq\gamma_{1,1}\leq\gamma_1$
(see Figure \ref{fig:CaseA1_increasing}).
We get:
\begin{align*}
\TV(f)^+ &= \abs{\gamma_{1} - \gamma_{1,1}} + a_{31}\abs{\gamma_{1,1}-\gamma_{1,0}} + a_{41}\abs{\gamma_{1,1}-\gamma_{1,0}}, \\
\Delta \TV (f) &= (a_{31}+a_{41}-1)\abs{\gamma_{1,1} - \gamma_{1,0}},\\
\Delta \Gamma &= \gamma_{1,1} - \gamma_{1,0}>0,\\
\Delta \bar{h} &= 0.
\end{align*}
Hence, (P2) holds and (P3) doesn't need to be verified. \\[10pt]
$-$ If $f(\rho_1)>f(\rho_{1,0})$, we define $\gamma_{1,1} = \dfrac{p_1}{p_2}\gamma_{2}$,
hence $\gamma_{1,0}\geq\gamma_{1,1}\geq\gamma_1$  (see Figure \ref{fig:CaseA1_decreasing} ).
We have:
\begin{align*}
\TV(f)^- &= \abs{\gamma_1 - \gamma_{1,0}},\\
\TV(f)^+ &= a_{31}\abs{\gamma_1 - \gamma_{1,0}}+ a_{41}\abs{\gamma_1 - \gamma_{1,0}},\\
\Delta \TV (f) &= (a_{31}+a_{41} -1)\abs{\gamma_1 - \gamma_{1,0}},\\
\Delta \Gamma &= (\gamma_1 - \gamma_{1,0})<0,\\
\Delta \bar{h} &= \dfrac{1}{p_1}(\gamma_1 - \gamma_{1,1})<0.
\end{align*}
Hence, (P2) and (P3) hold.
\begin{figure}[ht]
\begin{subfigure}[H]{0.3\textwidth}
	\centering
	\begin{tikzpicture}
\draw[<->](5,0)--(0,0)--(0,5);
\draw(0,0)rectangle(2,3);
\draw[red](0,0)--(1,3.5);
\draw(2,3)--(3,3);
\draw(3,0)--(3,3);
\draw(2,4)--(3,3)--(4.5,1.5);
\draw(4,0)--(4,2);
\node[left]at(0,3){$\gamma_2$};
\node[left]at(0,5){$\gamma^{\mathrm{inc}}_2$};
\node[right]at(2,4){$\gamma_3$};
\node[below]at(2,0){$\gamma_{1,0}$};
\node[below]at(3,0){$\gamma_{1,1}$};
\node[below]at(4,0){$\gamma_{1}$};
\node[below]at(5.2,0.15){$\gamma^{\mathrm{inc}}_{1}$};
\node[above]at(2,3){$Q_{0}$};
\node[right]at(3,3){$Q_{1}$};
\end{tikzpicture}
	\caption{Increasing $\gamma_1$}
	\label{fig:CaseA1_increasing}
\end{subfigure}\hspace{2cm} 
\begin{subfigure}[H]{0.3\textwidth}
	\centering
	\begin{tikzpicture}
\draw[<->](5,0)--(0,0)--(0,5);
\draw(0,0)rectangle(2,3);
\draw[red](0,0)--(2.3,3.5);
\draw(2,3)--(3,3);
\draw(3,0)--(3,3);
\draw(1,0)--(1,3);
\node[left]at(0,3){$\gamma_2$};
\node[below]at(2,0){$\gamma_{1,1}$};
\node[below]at(3,0){$\gamma_{1,0}$};
\node[below]at(1,0){$\gamma_{1}$};
\node[below]at(5.2,0.15){$\gamma^{\mathrm{inc}}_{1}$};
\node[left]at(0,5){$\gamma^{\mathrm{inc}}_2$};
\node[above]at(3,3){$Q_0$};
\node[above]at(1,3){$Q_1$};
\end{tikzpicture}
	\caption{Decreasing $\gamma_1$}
	\label{fig:CaseA1_decreasing}
\end{subfigure}
\caption{Case A1}
\label{fig:CaseA1}
\end{figure}

{\bf Case A2:} The incoming wave is $(\rho_2,\rho_{2,0})$ (on road 2). Since $\gamma_{2,0}$ is an active constraint,
$\rho_{2,0}\leq \rhocr$ and $\rho_2\leq \rhocr$. \\[10pt]
$-$If $f(\rho_2)>f(\rho_{2,0})$, we define
\[
\gamma_{2,1}=\dfrac{\gamma_3 - a_{31\gamma_1}}{a_{32}}
\quad\hbox{and}\quad
\gamma_{2,2} = \dfrac{p_2}{a_{31}p_1+a_{32}p_2}\, \gamma_3^{\max}
\]
so that $\gamma_{2,0}\leq\gamma_{2,1}\leq\gamma_{2,2}\leq\gamma_2$,
see Figure \ref{fig:CaseA2_increasing}. 
Note that this case is the same as in A1 except for the case in the drawing. In this case we have:
\begin{align*}
\TV(f)^- =&~ \abs{\gamma_2 - \gamma_{2,0}} = \abs{\gamma_{2}-\gamma_{2,2}}+\abs{\gamma_{2,2}-\gamma_{2,1}}+\abs{\gamma_{2,1}-\gamma_{2,0}},\\
\TV(f)^+ =&~ \dfrac{a_{32}}{a_{31}}\abs{\gamma_{2,2}-\gamma_{2,1}}+ \abs{\gamma_2 - \gamma_{2,2}}+ a_{32}\abs{\gamma_{2,1}- \gamma_{2,0}} \\
&+ \abs{a_{42} \abs{\gamma_{2,1} - \gamma_{2,0}}+\left(a_{42}-a_{41}\dfrac{a_{32}}{a_{31}}\right)\abs{\gamma_{2,2}-\gamma_{2,1}}},\\
\Delta \TV(f) =&~ \left( \dfrac{a_{32}}{a_{31}}-1 \right)\abs{\gamma_{2,2} - \gamma_{2,1}}+ (a_{32}-1)\abs{\gamma_{2,1}- \gamma_{2,0}} \\
&+ \abs{a_{42} \abs{\gamma_{2,1} - \gamma_{2,0}}+\left(a_{42}-a_{41}\dfrac{a_{32}}{a_{31}}\right)\abs{\gamma_{2,2}-\gamma_{2,1}}},\\
\Delta \Gamma =&~ (\gamma_{2,1} - \gamma_{2,0})+ \left(1- \dfrac{a_{32}}{a_{31}}\right)(\gamma_{2,2} - \gamma_{2,1}),\\
\Delta \bar{h} =&~ \dfrac{1}{p_2}(\gamma_{2,2} - \gamma_{2,0})
\end{align*}
Hence, (P2) holds while (P3) doesn't need to be checked.\\[10pt]
$-$ If $f(\rho_2)<f(\rho_{2,0})$ one has $\gamma_{2} < \gamma_{2,0}$, see Figure \ref{fig:CaseA2_decreasing}.
Therefore:
\begin{align*}
\TV(f)^- &= \abs{\gamma_2 - \gamma_{2,0}},\\
\TV(f)^+ &= (a_{32}+a_{42})\abs{\gamma_2 - \gamma_{2,0}},\\
\Delta \TV (f) &= 0 \quad\hbox{(recall that $a_{32}+a_{42}=1$)},\\
\Delta \Gamma &= (\gamma_2- \gamma_{2,0})<0,\\
\Delta \bar{h} &= \dfrac{1}{p_2}(\gamma_2 - \gamma_{2,0})<0.
\end{align*}
Hence, (P2) and (P3) hold. \\
\begin{figure}[ht]
\begin{subfigure}[H]{0.3\textwidth}
	\centering
	\begin{tikzpicture}
\draw[<->](5,0)--(0,0)--(0,5);
\draw(0,0)rectangle(3,2);
\draw[red](0,0)--(2.5,4);
\draw(0,3)--(3,3);
\draw(3,2)--(3,3);
\draw(1,4.3)--(1.5,4)--(3,3)--(4,2.2);
\draw(2.2,3.5)--(0,3.5);
\draw(1.5,4)--(0,4);
\node[left]at(0,3){$\gamma_{2,1}$};
\node[left]at(0,2){$\gamma_{2,0}$};
\node[left]at(0,3.5){$\gamma_{2,2}$};
\node[left]at(0,4){$\gamma_{2}$};
\node[below]at(5.2,0.15){$\gamma^{\mathrm{inc}}_{1}$};
\node[left]at(0,5){$\gamma^{\mathrm{inc}}_2$};
\node[right]at(3,2){$Q_0$};
\node[above]at(2.1,3.7){$Q_1$};
\node[right]at(4,2.2){$\gamma_{3}$};
\end{tikzpicture}
	\caption{Increasing $\gamma_2$}
	\label{fig:CaseA2_increasing}
\end{subfigure}\hspace{2cm} 
\begin{subfigure}[H]{0.3\textwidth}
	\centering
	\begin{tikzpicture}
\draw[<->](5,0)--(0,0)--(0,5);
\draw(0,0)rectangle(3,2);
\draw[red](0,0)--(2.5,4);
\draw(0,1)--(3,1);
\node[left]at(0,2){$\gamma_{2,0}$};
\node[left]at(0,1){$\gamma_{2}$};
\node[below]at(5.2,0.15){$\gamma^{\mathrm{inc}}_{1}$};
\node[left]at(0,5){$\gamma^{\mathrm{inc}}_2$};
\node[right]at(3,2){$Q_0$};
\node[right]at(3,1){$Q_1$};
\end{tikzpicture}
	\caption{Decreasing $\gamma_12$}
	\label{fig:CaseA2_decreasing}
\end{subfigure}
\caption{Case A2}
\label{fig:CaseA2}
\end{figure}
{\bf Case A3:} The incoming wave is $(\rho_{3,0},\rho_{3})$ (on road 3, the case of road 4 being similar). \\[10pt]
$-$ If $f(\rho_3) < f(\rho_{3,0})$ We define $\gamma_{3,1} = \left(a_{31}\dfrac{p_1}{p_2} +a_{32}\right) \gamma_{2}$ so that $\gamma_{3}\leq \gamma_{3,1}\leq \gamma_{3,0}$ (see Figure \ref{fig:CaseA3_decreasing}):
\begin{align*}
\TV(f)^- &= \abs{\gamma_{3} - \gamma_{3,0}},\\
\TV(f)^+ &= \dfrac{a_{41}+1}{a_{31}}\abs{\gamma_{3,1}-\gamma_{3,0}}
+ \dfrac{(a_{41}+1)p_1 + (a_{42}+1)p_2}{a_{31}p_1 + a_{32}p_2}  \abs{\gamma_3 - \gamma_{3,1}},\\
\Delta \TV(f) &= \dfrac{a_{41}}{a_{31}} \abs{\gamma_{3,1}-\gamma_{3,0}}+ 
\left(\dfrac{(a_{41}+1)p_1 + (a_{42}+1)p_2}{a_{31}p_1 + a_{32}p_2}  -1\right) \abs{\gamma_3 - \gamma_{3,1}},\\
\Delta \Gamma &= \dfrac{1}{a_{31}}(\gamma_{3,1} - \gamma_{3,0})+ \dfrac{a_{41}p_1 + a_{42}p_2}{a_{31}p_1 + a_{32}p_2} (\gamma_3 - \gamma_{3,1})<0,\\
\Delta \bar{h} &= \dfrac{1}{a_{31}p_1+a_{32}p_2}(\gamma_3 - \gamma_{3,1})<0. 
\end{align*}
Hence, (P2) and (P3) hold. \\[10pt]
\begin{figure}[ht]
\centering
\begin{tikzpicture}
\draw[<->](6,0)--(0,0)--(0,6);
\draw(0,0)rectangle(2,3);
\draw[red](0,0)--(2,3)--(2.8,4);
\draw(2,0)--(0,3);
\draw(4,0)--(2,3)--(0.3,5.5);
\draw(5,0)--(3,3)--(1.7,5);
\draw(2,3)--(3,3)--(3,0);
\node[right]at(3,3){$\gamma_{3,0}$};
\node[above]at(4,0.5){$\gamma_{3,1}$};
\node[above]at(1.1,0.5){$\gamma_{3}$};
\node[below]at(6.2,0.15){$\gamma^{\mathrm{inc}}_{1}$};
\node[left]at(0,6){$\gamma^{\mathrm{inc}}_2$};
\node[above]at(1,1.65){$Q_1$};
\node[left]at(3,2.65){$Q_0$};
\end{tikzpicture}
\caption{Case a3 -  Decreasing $\gamma_3$}
\label{fig:CaseA3_decreasing}
\end{figure}
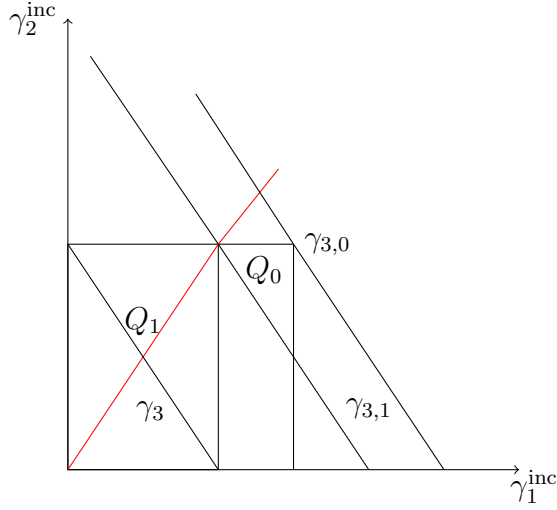
$-$ If $f(\rho_3) > f(\rho_{3,0})$ and $\gamma_{3} > \gamma_{3,0}$ then we stay demand constrained and nothing happens. 
%
\item {\bf Case B: Supply constrained, priority line intersects a demand constraint.} 
Even in this case, it is not restrictive to assume that the priority line $hP$, $h>0$,
intersects the constraint $\gamma_1^\mathrm{inc}=\gamma_{1,0}$. We spit the proof in  several subcases
depending on the origin of the incoming wave: \\[10pt]
{\bf Case B1:} The incoming wave is $(\rho_1,\rho_{1,0})$ (on road 1). 
We define $\gamma_{1,1}=\dfrac{a_{42}-a_{32}}{a_{31}a_{42} - a_{32}a_{41}}\ \gamma_1$
(see Figure \ref{fig:CaseB1_increasing} and \ref{fig:CaseB1_decreasing}). We distinguish the two situations:\\[10pt]
$-$ If $f(\rho_1)>f(\rho_{1,0})$ we get $\gamma_{1,0}\leq \gamma_{1,1}\leq \gamma_{1,2}\leq \gamma_1$,
where $\gamma_{1,2}=\dfrac{p_1}{a_{41}p_1 + a_{42}p_2}\,\gamma_4$
(see Figure \ref{fig:CaseB1_increasing}). Then we have:
\begin{align*}
\TV(f)^- =&~ \abs{\gamma_1 - \gamma_{1,0}},\\
\TV(f)^+ =&~ \abs{\gamma_1 - \gamma_{1,2}} 
+\left( \dfrac{a_{41}}{a_{42}} +\abs{a_{31}- a_{32}\dfrac{a_{41}}{a_{42}}} \right)\abs{\gamma_{1,2}-\gamma_{1,1}} \\
&+\left( \dfrac{a_{31}}{a_{32}} + \abs{a_{41}- a_{42}\dfrac{a_{31}}{a_{32}}} \right)\abs{\gamma_{1,1}-\gamma_{1,0}},\\
\Delta \TV(f) =&~ \left( \dfrac{a_{41}}{a_{42}} +\abs{a_{31}- a_{32}\dfrac{a_{41}}{a_{42}}} -1\right)\abs{\gamma_{1,2}-\gamma_{1,1}}\\
&+\left( \dfrac{a_{31}}{a_{32}} + \abs{a_{41}- a_{42}\dfrac{a_{31}}{a_{32}}} -1\right)\abs{\gamma_{1,1}-\gamma_{1,0}},\\
\Delta \Gamma =&~ \left(1 - \dfrac{a_{31}}{a_{32}}\right) (\gamma_{1,1}-\gamma_{1,0})+ \left(1-\dfrac{a_{41}}{a_{42}}\right)(\gamma_{1,2}-\gamma_{1,1}),\\
\Delta \bar{h} =&~ \dfrac{1}{p_1} (\gamma_{1,2} - \gamma_{1,0}).
\end{align*}
Hence, (P2) holds and (P3) doesn't need to be checked. \\[10pt]
$-$ If $f(\rho_1) < f(\rho_{1,0})$ and $\gamma_{1,0}\geq \gamma_{1,1}\geq \gamma_{1,2}\geq \gamma_1$
(see Figure \ref{fig:CaseB1_decreasing}), we define $\gamma_{1,2} = \dfrac{\gamma_3-a_{32}\gamma_2}{a_{31}}$ and we get:
\begin{align*}
\TV(f)^- =&~ \abs{\gamma_1 - \gamma_{1,0}},\\
\TV(f)^+ =&~ (a_{31} +a_{41})\abs{\gamma_1-\gamma_{1,2}}
+ \left( \dfrac{a_{31}}{a_{32}} + \abs{a_{41}-a_{42}\dfrac{a_{31}}{a_{32}}}\right) \abs{\gamma_{1,2} - \gamma_{1,1}} \\
&+ \left( \dfrac{a_{41}}{a_{42}}+\abs{a_{31}-a_{32}\dfrac{a_{41}}{a_{42}}}\right) \abs{\gamma_{1,1} - \gamma_{1,0}},\\
\Delta \TV(f) = &
 \left( \dfrac{a_{31}}{a_{32}} + \abs{a_{41}-a_{42}\dfrac{a_{31}}{a_{32}}} -1\right) \abs{\gamma_{1,2} - \gamma_{1,1}} \\
&+ \left( \dfrac{a_{41}}{a_{42}}+\abs{a_{31}-a_{32}\dfrac{a_{41}}{a_{42}}} -1\right) \abs{\gamma_{1,1} - \gamma_{1,0}},\\
\Delta \Gamma =&~ (\gamma_1 - \gamma_{1,2})+ \left(1 - \dfrac{a_{31}}{a_{32}}\right)(\gamma_{1,2} - \gamma_{1,1})
+\left(1-\dfrac{a_{41}}{a_{42}}\right)(\gamma_{1,1} - \gamma_{1,0}),\\
\Delta \bar{h} =&~\dfrac{1}{p_1}(\gamma_1 - \gamma_{1,0})<0.
\end{align*}
Hence, (P2) and (P3) hold. \\[10pt]
\begin{figure}[ht]
\begin{subfigure}[H]{0.3\textwidth}
	\centering
	\begin{tikzpicture}
\draw[<->](6,0)--(0,0)--(0,6);
\draw(0,3)--(1,3);
\draw(2,0)--(2,2.5);
\draw[red](0,0)--(5,2);
\draw(5,1)--(0,3.5);
\draw(4,0)--(1.5,4);
\draw(2.65,0)--(2.65,2.2);
\draw(3.2,0)--(3.2,1.3);
\draw(3.8,0)--(3.8,0.3);
\node[below]at(1.6,0){$\gamma_{1,0}$};
\node[below]at(2.5,0){$\gamma_{1,1}$};
\node[below]at(3.2,0){$\gamma_{1,2}$};
\node[below]at(3.8,0){$\gamma_{1}$};
\node[below]at(0.3,4){$\gamma_{3}$};
\node[below]at(1.9,4){$\gamma_{4}$};
\node[below]at(6.2,0.15){$\gamma^{\mathrm{inc}}_{1}$};
\node[left]at(0,6){$\gamma^{\mathrm{inc}}_2$};
\node[above]at(2,2.45){$Q_0$};
\node[above]at(3.7,0.8){$Q_1$};
\end{tikzpicture}
	\caption{Increasing $\gamma_1$}
	\label{fig:CaseB1_increasing}
\end{subfigure}\hspace{3cm} 
\begin{subfigure}[H]{0.3\textwidth}
	\centering
	\begin{tikzpicture}
\draw[<->](6,0)--(0,0)--(0,6);
\draw(0,3)--(1.6,3);
\draw(1.6,0)--(1.6,3);
\draw[red](0,0)--(5,1.9);
\draw(5,2)--(0,3.5);
\draw(5,1)--(1,4);
\draw(2.8,0)--(2.8,2.6);
\draw(0.8,0)--(0.8,3);
\draw(3.5,0)--(3.5,2.1);
\node[below]at(3.5,0){$\gamma_{1,0}$};
\node[below]at(2.7,0){$\gamma_{1,1}$};
\node[below]at(1.6,0){$\gamma_{1,2}$};
\node[below]at(0.8,0){$\gamma_{1}$};
\node[above]at(1,4){$\gamma_{4}$};
\node[above]at(1,3.2){$\gamma_{3}$};
\node[below]at(6.2,0.15){$\gamma^{\mathrm{inc}}_{1}$};
\node[left]at(0,6){$\gamma^{\mathrm{inc}}_2$};
\node[above]at(3.2,1.7){$Q_0$};
\node[above]at(0.5,2.4){$Q_1$};
\end{tikzpicture}
	\caption{Decreasing $\gamma_1$}
	\label{fig:CaseB1_decreasing}
\end{subfigure}
\caption{Case B1}
\label{fig:CaseB1}
\end{figure}
{\bf Case B2:} The incoming wave is $(\rho_2,\rho_{2,0})$ (on road 2). \\[10pt]
$-$ If $f(\rho_2)>f(\rho_{2,0})$: nothing happens (see Figure \ref{fig:CaseB2_increasing}). \\[10pt]
$-$If $f(\rho_2)<f(\rho_{2,0})$ (see Figure \ref{fig:CaseB2_decreasing})
we define $\gamma_{2,1}=\dfrac{p_1}{p_2}\gamma_{1}$ and we compute:
\begin{align*}
\TV(f)^- =&~ \abs{\gamma_2 - \gamma_{2,0}},\\
\TV(f)^+ =&~ (a_{32} + a_{42})\abs{\gamma_2 - \gamma_{2,0}},\\
\Delta \TV(f) =&~  (a_{32}+a_{42}-1)\abs{\gamma_2 - \gamma_{2,0}}=0,\\
\Delta \Gamma =&~ (\gamma_2 - \gamma_{2,0})<0,\\
\Delta \bar{h} =&~ \dfrac{1}{p_2}(\gamma_2 - \gamma_{2,1})<0.
\end{align*}
Hence, (P2) and (P3) hold. \\[10pt]
\begin{figure}[ht]
\begin{subfigure}[H]{0.3\textwidth}
	\centering
\begin{tikzpicture}
\draw[<->](6,0)--(0,0)--(0,6);
\draw(0,3)--(1,3);
\draw(2,0)--(2,2.5);
\draw[red](0,0)--(5,2);
\draw(5,1)--(0,3.5);
\draw(0,2.5)--(2,2.5);
\node[below]at(2,0){$\gamma_{1}$};
\node[left]at(0,2.5){$\gamma_{2,0}$};
\node[left]at(0,3){$\gamma_{2}$};
\node[below]at(0.3,4){$\gamma_{3}$};
\node[below]at(6.2,0.15){$\gamma^{\mathrm{inc}}_{1}$};
\node[left]at(0,6){$\gamma^{\mathrm{inc}}_2$};
\node[above]at(2.4,2.5){$Q_0=Q_1$};
\end{tikzpicture}
	\caption{Increasing $\gamma_2$}
	\label{fig:CaseB2_increasing}
\end{subfigure}\hspace{3cm} 
\begin{subfigure}[H]{0.3\textwidth}
	\centering
	\begin{tikzpicture}
\draw[<->](6,0)--(0,0)--(0,6);
\draw(0,1.6)--(2,1.6);
\draw(2,0)--(2,2.5);
\draw[red](0,0)--(5,4);
\draw(5,1)--(0,3.5);
\draw(0,2.5)--(2,2.5);
\draw(0,1)--(2,1);
\node[below]at(2,0){$\gamma_{1}$};
\node[left]at(0,2.5){$\gamma_{2,0}$};
\node[left]at(0,1){$\gamma_{2}$};
\node[left]at(0,1.6){$\gamma_{2,1}$};
\node[below]at(0.3,4){$\gamma_{3}$};
\node[below]at(6.2,0.15){$\gamma^{\mathrm{inc}}_{1}$};
\node[left]at(0,6){$\gamma^{\mathrm{inc}}_2$};
\node[above]at(2,2.5){$Q_0$};
\node[right]at(2,1){$Q_1$};
\end{tikzpicture}
	\caption{Decreasing $\gamma_2$}
	\label{fig:CaseB2_decreasing}
\end{subfigure}
\caption{Case B2}
\label{fig:CaseB2}
\end{figure}
 {\bf Case B3:} The incoming wave is $(\rho_{3,0},\rho_3)$ (on road 3). \\[10pt]
$-$ If $f(\rho_3) > f(\rho_{3,0})$ (see Figure \ref{fig:CaseB3_increasing}), 
we define 
\[ \gamma_{3,1} = \min \left\lbrace
\begin{array}{ll}
a_{31}\gamma_{1}+ a_{32}\gamma^{\max}_2, \\
a_{31}\gamma_{1}+ \dfrac{a_{32}}{a_{42}}(\gamma_{4}^{\max}-a_{41}\gamma_{1}).
\end{array}
\right.\]
We get: 
\begin{align*}
\TV(f)^- &= \abs{\gamma_3 - \gamma_{3,0}},\\
\TV(f)^+ &= \left(\dfrac{1}{a_{32}}+ \dfrac{a_{42}}{a_{32}}\right) \abs{\gamma_{3,1} - \gamma_{3,0}}+\abs{\gamma_3 - \gamma_{3,1}},\\
\Delta \TV(f) &= \left(\dfrac{1}{a_{32}}+\dfrac{a_{42}}{a_{32}}-1\right)\abs{\gamma_{3,1} - \gamma_{3,0}},\\
\Delta\Gamma &= \dfrac{1}{a_{32}}(\gamma_{3,1}- \gamma_{3,0})>0,\\
\Delta \bar{h} &= 0.
\end{align*}
Hence, (P2) holds and (P3) does not need to be checked. \\[10pt]
$-$ If $f(\rho_3) < f(\rho_{3,0})$  (see Figure \ref{fig:CaseB3_decreasing}),
we define $\gamma_{3,1} = \left(a_{31}+a_{32}\dfrac{p_2}{p_1}\right) \gamma_{1}$ and we compute:
\begin{align*}
\TV(f)^- =&~ \abs{\gamma_3 - \gamma_{3,0}},\\
\TV(f)^+ =&~ \dfrac{p_1}{a_{31}p_1+a_{32}p_2}\abs{\gamma_3 - \gamma_{3,1}} + \dfrac{1}{a_{32}}\abs{\gamma_{3,1}-\gamma_{3,0}}+ \dfrac{p_2}{a_{31}p_1+a_{32}p_2}\abs{\gamma_3 - \gamma_{3,1}} \\
&+ \dfrac{a_{42}}{a_{32}}\abs{\gamma_{3,1} - \gamma_{3,0}} + \dfrac{a_{41}p_1+a_{42}p_2}{a_{31}p_1+a_{32}p_2}\abs{\gamma_3 - \gamma_{3,1}} ,\\
\Delta \TV(f) =&~\left( \dfrac{(1+a_{41})p_1+(1+a_{42})p_2}{a_{31}p_1+a_{32}p_2} -1\right) \abs{\gamma_3 - \gamma_{3,1}} 
+ \left(\dfrac{1+a_{42}}{a_{32}}-1\right) \abs{\gamma_{3,1}-\gamma_{3,0}},\\
\Delta \Gamma =&~  \dfrac{p_1+p_2}{a_{31}p_1+a_{32}p_2}(\gamma_3 - \gamma_{3,1}) + \dfrac{1}{a_{32}}(\gamma_{3,1} - \gamma_{3,0})<0,\\
\Delta \bar{h} =&~ \dfrac{1}{a_{31}p_1+a_{32}p_2}(\gamma_3 - \gamma_{3,1})<0.
\end{align*}
Hence, (P2) and (P3) hold. \\[10pt]
\begin{figure}[ht]
\begin{subfigure}[H]{0.3\textwidth}
	\centering
	\begin{tikzpicture}
\draw[<->](6,0)--(0,0)--(0,6);
\draw(0,2)--(2,2);
\draw(2,0)--(2,2.5);
\draw[red](0,0)--(5,2);
\draw(5,1)--(0,3.5);
\draw(5,2.5)--(0,5);
\draw(5,0.5)--(0,3);
\draw(0,2.5)--(2,2.5);
\draw(0,4.5)--(2,2.5)--(4.5,0);
\node[below]at(2,0){$\gamma_{1}$};
\node[left]at(0,2){$\gamma_{2,0}$};
\node[left]at(0,2.5){$\gamma_{2}^{\max}$};
\node[above]at(0.3,5){$\gamma_{3}$};
\node[above]at(0.3,3.3){$\gamma_{3,1}$};
\node[above]at(5.3,0.3){$\gamma_{3,0}$};
\node[above]at(0.3,4.2){$\gamma_{4}$};
\node[below]at(6.2,0.15){$\gamma^{\mathrm{inc}}_{1}$};
\node[left]at(0,6){$\gamma^{\mathrm{inc}}_2$};
\node[above]at(1.7,1.47){$Q_0$};
\node[above]at(2,2.5){$Q_1$};
\end{tikzpicture}
	\caption{Increasing $\gamma_3$}
	\label{fig:CaseB3_increasing}
\end{subfigure}\hspace{3cm} 
\begin{subfigure}[H]{0.3\textwidth}
	\centering
	\begin{tikzpicture}
\draw[<->](6,0)--(0,0)--(0,6);
\draw(0,3)--(1,3);
\draw(2,0)--(2,2.5);
\draw[red](0,0)--(5,2);
\draw(5,1)--(0,3.5);
\draw(3.5,0)--(2,0.8)--(0,1.8);
\draw(2,0)--(0,1);
\node[below]at(2,0){$\gamma_{1,0}$};
\node[above]at(0.5,1.5){$\gamma_{3,1}$};
\node[above]at(0.5,0.75){$\gamma_{3}$};
\node[left]at(0,3){$\gamma_{2}$};
\node[below]at(0.5,3.8){$\gamma_{3,0}$};
\node[below]at(6.2,0.15){$\gamma^{\mathrm{inc}}_{1}$};
\node[left]at(0,6){$\gamma^{\mathrm{inc}}_2$};
\node[above]at(2,2.5){$Q_0$};
\node[above]at(1.1,0.5){$Q_1$};
\end{tikzpicture}
	\caption{Decreasing $\gamma_3$}
	\label{fig:CaseB3_decreasing}
\end{subfigure}
\caption{Case B3}
\label{fig:CaseB3}
\end{figure}
 {\bf Case B4:} The incoming wave is $(\rho_{4,0},\rho_4)$ (on road 4). \\[10pt]
$-$ If $f(\rho_4) > f(\rho_{4,0})$ nothing changes.  \\[10pt]
$-$ If $f(\rho_4) < f(\rho_{4,0})$:  the same as case {\bf B3} decreasing exchanging the roles of $\gamma_3$ and $\gamma_4$.
\item {\bf Case C: Supply constrained, priority line intersects a supply constraint.} 
It is not restrictive to assume that the priority line $hP$, $h>0$,
intersects the constraint $a_{31}\gamma_1^\mathrm{inc}+a_{32}\gamma_2^\mathrm{inc}=\gamma_{3,0}$. 
We distinguish the following subcases,
depending on the origin of the incoming wave: \\[10pt]
 {\bf Case C1:} The incoming wave is $(\rho_1,\rho_{1,0})$ (on road 1). \\[10pt]
$-$ If $f(\rho_1) > f(\rho_{1,0})$, nothing changes. \\[10pt]
$-$ If $f(\rho_1) < f(\rho_{1,0})$, the analysis is similar to case {\bf B1} decreasing (see Figure \ref{fig:CaseC1_decreasing}). \\[10pt]
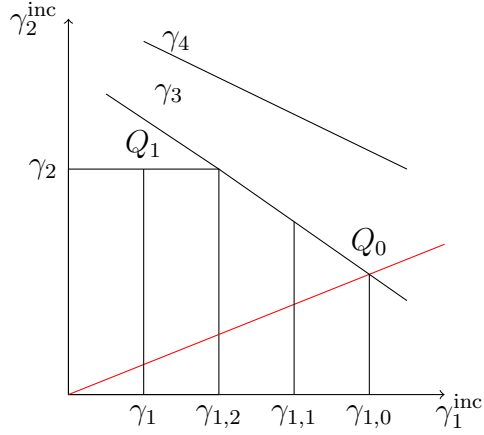
\begin{figure}[ht]
\centering
\begin{tikzpicture}
\draw[<->](5,0)--(0,0)--(0,5);
\draw(0,0)rectangle(2,3);
\draw[red](0,0)--(5,2);
\draw(3,0)--(3,2.3);
\draw(1,0)--(1,3);
\draw(0.5,4)--(2,3)--(4.5,1.25);
\draw(1,4.7)--(4.5,3);
\draw(4,0)--(4,1.6);
\node[left]at(0,3){$\gamma_2$};
\node[right]at(1,4){$\gamma_3$};
\node[below]at(1,0){$\gamma_{1}$};
\node[below]at(2,0){$\gamma_{1,2}$};
\node[below]at(3,0){$\gamma_{1,1}$};
\node[below]at(4,0){$\gamma_{1,0}$};
\node[below]at(5.2,0.15){$\gamma^{\mathrm{inc}}_{1}$};
\node[left]at(0,5){$\gamma^{\mathrm{inc}}_2$};
\node[right]at(1.1,4.7){$\gamma_4$};
\node[above]at(4,1.7){$Q_0$};
\node[above]at(1,3){$Q_1$};
\end{tikzpicture}
\caption{Case C1 -  decreasing $\gamma_1$}
\label{fig:CaseC1_decreasing}
\end{figure}
{\bf Case C2:} The incoming wave is $(\rho_2,\rho_{2,0})$ (on road 2). This case is symmetric to {\bf C1}. \\[10pt]
{\bf Case C3:} The incoming wave is $(\rho_{3,0},\rho_{3})$ (on road 3). \\[10pt]
$-$ If $f(\rho_3) > f(\rho_{3,0})$ (see Figure \ref{fig:CaseC3_increasing}),
we define 
\[
\gamma_{3,1} = \min \left\lbrace
\begin{array}{ll}
\left(a_{31}+ a_{32}\dfrac{p_2}{p_1}\right)\gamma^{\max}_1, \\[10pt]
\left(a_{31}\dfrac{p_1}{p_2}+ a_{32}\right)\gamma_{2}^{\max},\\[10pt]
\dfrac{a_{31}p_1+a_{32}p_2}{a_{41}p_1+a_{42}p_2} \ \gamma_4^{\max},
\end{array}
\right.
\]
and we compute: 
\begin{align*}
\TV(f)^- &= \abs{\gamma_3 - \gamma_{3,0}},\\
\TV(f)^+ &= \dfrac{p_1+p_2}{a_{31}p_1+a_{32}p_2}\abs{\gamma_{3,1}-\gamma_{3,0}}+ \abs{\gamma_3-\gamma_{3,1}}+ \dfrac{a_{41}p_1+a_{42}p_2}{a_{31}p_1+a_{32}p_2}\abs{\gamma_{3,1}-\gamma_{3,0}},\\
\Delta \TV(f) &= \left(\dfrac{(a_{41}+1)p_1+(a_{42}+1)p_2}{a_{31}p_1+a_{32}p_2}-1\right)\abs{\gamma_{3,1}-\gamma_{3,0}},\\
\Delta\Gamma &=  \dfrac{p_1+p_2}{a_{31}p_1+a_{32}p_2}(\gamma_{3,1} - \gamma_{3,0})>0,\\
\Delta \bar{h} &= \dfrac{1}{a_{31}p_1+a_{32}p_2}(\gamma_{3,1}-\gamma_{3,0})>0.
\end{align*}
Hence, (P2) holds and (P3) does not count since we are increasing fluxes. \\[10pt]
$-$ If $f(\rho_3) < f(\rho_{3,0})$  (see Figure \ref{fig:CaseC3_decreasing}) we get:
\begin{align*}
\TV(f)^- &= \abs{\gamma_3 - \gamma_{3,0}}, \\
\TV(f)^+ &= \dfrac{p_1 + p_2}{a_{31}p_1+a_{32}p_2}\abs{\gamma_3 -\gamma_{3,0}} + \dfrac{a_{41}p_1 + a_{42}p_2}{a_{31}p_1+a_{32}p_2}\abs{\gamma_{3} - \gamma_{3,0}}, \\
\Delta \TV(f) &= \left(\dfrac{(a_{41}+1)p_1+(a_{42}+1)p_2}{a_{31}p_1+a_{32}p_2}-1\right)\abs{\gamma_{3}-\gamma_{3,0}},\\
\Delta\Gamma &= \dfrac{p_1 + p_2}{a_{31}p_1+a_{32}p_2}(\gamma_3 -\gamma_{3,0})<0,\\
\Delta \bar{h} &= \dfrac{1}{a_{31}p_1+a_{32}p_2}(\gamma_3 - \gamma_{3,0})<0.
\end{align*}
Hence, (P2) and (P3) hold. \\[10pt]
\begin{figure}[ht]
\begin{subfigure}[H]{0.3\textwidth}
	\centering
	\begin{tikzpicture}
\draw[<->](5,0)--(0,0)--(0,5);
\draw(0,0)rectangle(2,3);
\draw[red](0,0)--(2,3)--(3.5,5);
\draw(1.2,5)--(4.5,2.75);
\draw(1,5)--(2,3)--(3,1);
\draw(0.5,4)--(2,3)--(4.5,1.25);
\draw(0,2)--(2.8,0);
\node[left]at(0,3){$\gamma_2$};
\node[right]at(0.8,3.2){$\gamma_{3,1}$};
\node[below]at(2,0){$\gamma_{1}$};
\node[below]at(1.5,1.5){$\gamma_{3,0}$};
\node[below]at(4,3){$\gamma_{3}$};
\node[below]at(3.2,1.5){$\gamma_{4}$};
\node[below]at(5.2,0.15){$\gamma^{\mathrm{inc}}_{1}$};
\node[left]at(0,5){$\gamma^{\mathrm{inc}}_2$};
\node[above]at(0.8,1.5){$Q_0$};
\node[above]at(2.7,4.1){$Q_1$};
\end{tikzpicture}
	\caption{Increasing $\gamma_3$}
	\label{fig:CaseC3_increasing}
\end{subfigure}\hspace{3cm} 
\begin{subfigure}[H]{0.3\textwidth}
	\centering
	\resizebox{1.6\textwidth}{!}{
	\begin{tikzpicture}[baseline=0.2cm]
\draw[<->](6,0)--(0,0)--(0,6);
\draw[red](0,0)--(5,2);
\draw(5,1)--(0,3.5);
\draw(2,0)--(0,1);
\node[above]at(0.5,0.75){$\gamma_{3}$};
\node[below]at(0.5,3.8){$\gamma_{3,0}$};
\node[below]at(6.2,0.15){$\gamma^{\mathrm{inc}}_{1}$};
\node[left]at(0,6){$\gamma^{\mathrm{inc}}_2$};
\node[above]at(4,1.5){$Q_0$};
\node[above]at(1.2,0.5){$Q_1$};
\end{tikzpicture}}
	\caption{Decreasing $\gamma_3$}
	\label{fig:CaseC3_decreasing}
\end{subfigure}
\caption{Case C3}
\label{fig:CaseC3}
\end{figure}
{\bf Case C4:} The incoming wave is $(\rho_{4,0},\rho_{4})$ (on road 4). \\[10pt]
$-$ If $f(\rho_4) > f(\rho_{4,0})$,  nothing happens. \\[10pt]
$-$ If $f(\rho_4) < f(\rho_{4,0})$, the situation is similar to case {\bf C3} with the roles of $\gamma_3$ and $\gamma_4$ reversed.
\end{itemize}

\bibliographystyle{plain}
\nocite{*}
\bibliography{RSbiblio}

\end{document}